\documentclass{amsart}
\usepackage{amsmath}
\usepackage{amssymb}
\usepackage{amsthm}

\usepackage{hyperref}

\usepackage{graphicx}
\usepackage{overpic}
\usepackage{xcolor}
\usepackage{mathtools}
\usepackage{overpic}

\usepackage{tikz}
\usetikzlibrary{matrix}

\newtheorem{theorem}{Theorem}[section]
\newtheorem{lemma}[theorem]{Lemma}

\newtheorem{proposition}[theorem]{Proposition}

\newtheorem{conjecture}[theorem]{Conjecture}

\newtheorem{introtheorem}{Theorem}

\theoremstyle{definition}
\newtheorem{definition}[theorem]{Definition}

\theoremstyle{remark}
\newtheorem{remark}{Remark}[theorem]

\usepackage{marginnote}
\newcounter{mynote}

\title{Connectivity of Coxeter group Morse boundaries}

\author[M. Cordes]{Matthew Cordes}
    \address{Department of Mathematics, Heriot-Watt University and Maxwell Institute for Mathematical Sciences, Edinburgh, United Kingdom}
    \email{m.cordes@hw.ac.uk}

\author[I. Levcovitz]{Ivan Levcovitz}
    \address{}
    \email{ivan.levcovitz@gmail.com}

\begin{document}
\begin{abstract}
    We study the connectivity of Morse boundaries of Coxeter groups. 
    We define two conditions on the defining graph of a Coxeter group: wide-avoidant and wide-spherical-avoidant. 
    We show that wide-spherical-avoidant, one-ended, affine-free Coxeter groups have connected and locally connected Morse boundaries.
    On the other hand, one-ended Coxeter groups that are not wide-avoidant and not wide have disconnected Morse boundary. 
    For the right-angled case, we get a full characterization: 
    a one-ended right-angled Coxeter group has connected, non-empty Morse boundary if and only if it is wide-avoidant. 
    Along the way we characterize Morse geodesic rays in affine-free Coxeter groups as those that spend uniformly bounded time in cosets of wide special subgroups. 
\end{abstract}

	\maketitle

\section{Introduction}
Boundaries play a fundamental role in the study of (Gromov) hyperbolic spaces and groups. However, there are still many metric spaces and groups which are not hyperbolic 
that, nonetheless, display some hyperbolic behavior.  
The Morse boundary was introduced by Charney--Sultan \cite{Charney-Sultan} and the first author \cite{Cordes2017} with the goal of identifying and encoding this behavior in a useful way.
Roughly, it collects all the geodesics which have hyperbolic-like behavior into a boundary; these geodesics are called Morse geodesics. 
The key property of this boundary is that it is a quasi-isometry invariant, meaning the Morse boundary of a finitely generated group $G$ is well-defined. 

Despite the Morse boundary being a fruitful quasi-isometry invariant, the Morse boundaries of many groups are difficult to fully describe.
Even basic topological properties like connectedness or local connectedness are hard to show.  
Coxeter groups are a classical set of groups defined by abstracting properties of reflection groups.
In this article, we study the connectivity of Morse boundaries for Coxeter groups.

At odds to hyperbolic-like behavior in a group is ``flat-like" behavior. 
Introduced by Behrstock--Drutu--Mosher in \cite{Behrstock-Drutu-Mosher}, \emph{wide groups} are a set of groups that have no Morse geodesics, i.e., they have flat-like behavior. 
Identifying wide subgroups in general can be difficult, but Behrstock--Hagen--Sisto in \cite{Behrstock-Hagen-Sisto} characterize wide Coxeter groups.
In order to understand the Morse directions of Coxeter groups, we study the interplay between their wide special subgroups.

This interplay is captured in the notion of a \emph{wide-avoidant} Coxeter group (Definition~\ref{def:wide-avoidant}), which we introduce in this paper. 
Loosely, a Coxeter group is wide-avoidant if given any two vertices in its defining graph and a special wide subgroup, then there is a path between these two vertices that avoids the subgraph corresponding to the wide subgroup.
Failure of this condition is enough to ensure that the Morse boundary of a Coxeter group is either empty or disconnected:

\begin{introtheorem} \label{thm:intro_not_wide_avoidant}
    A Coxeter group that is not wide and is not wide-avoidant has disconnected Morse boundary. Moreover, it admits a splitting either over a finite special 
    subgroup or over a special subgroup that is contained in a wide special subgroup. 
\end{introtheorem}

In fact, we conjecture the full characterization:
\begin{conjecture} \label{conj:wide-avoidant}
A Coxeter group has connected, non-empty Morse boundary if and only if it is one-ended and wide-avoidant.
\end{conjecture}
We prove this conjecture for right-angled Coxeter groups and certain classes of Coxeter groups. The right-angled Coxeter characterization being:
\begin{introtheorem}
Let $W_\Gamma$ be a right-angled Coxeter group. Then exactly one of the following holds:
\begin{enumerate}
    \item $W_\Gamma$ has empty Morse boundary and is finite or wide,
    \item $W_\Gamma$ has disconnected Morse boundary and has more than one end (either splits over a finite special subgroup or is virtually $\mathbb{Z}$)
    \item $W_\Gamma$ has connected and locally connected Morse boundary and is wide-avoidant
    \item $W_\Gamma$ has disconnected Morse boundary, is not wide-avoidant and admits a non-trivial splitting $W_\Gamma = W_{\Gamma_1} \star_{W_\Delta} W_{\Gamma_2}$ where $W_\Delta$ is an infinite subgroup  of a wide special subgroup.
\end{enumerate}
\end{introtheorem}

We were not able to prove Conjecture~\ref{conj:wide-avoidant} for all Coxeter groups. 
However, we give the following sufficient conditions for the connectivity of the Morse boundary. 
\begin{introtheorem} \label{thm:intro_connected}
    An affine-free, one-ended, wide-spherical-avoidant Coxeter group has connected and locally connected Morse boundary. 
\end{introtheorem}

Recall that a Coxeter group is \emph{affine-free} if it has no irreducible, special, affine subgroup of rank $3$ or larger. 
We define a Coxeter group to be \emph{wide-spherical-avoidant} 
if, loosely, given a certain type of subgraph corresponding to a wide subgroup and a restricted type of spherical subgroup, then there is a path between any two vertices that avoids this subgraph. 
See Definition~\ref{def:wide-spherical-avoidant} for a precise definition.

Theorem~\ref{thm:intro_connected} provides what we believe to be the first examples of non-relatively hyperbolic groups with locally-connected Morse boundary (e.g. Figure \ref{fig:non-rel-conn-mb}).
In fact, the only other known class of non-hyperbolic groups with locally-connected Morse boundary are those quasi-isometric to 
fundamental groups of finite-volume, cusped hyperbolic $n$-manifolds \cite{Charney-Cordes-Sisto}\cite{FKSZ} (which are relatively hyperbolic). 
Note that the right-angled Coxeter defined on the 1-skeleton of a 3-cube is virtually a finite-volume, 
cusped hyperbolic $3$-manifold group \cite{Haulmark-Nguyen-Tran}. 
\begin{figure}\label{fig:non-rel-conn-mb}
    \begin{overpic}[scale=.5]{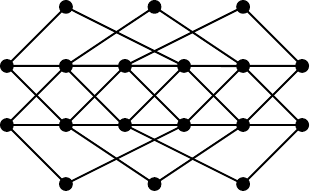}
  \end{overpic}  
  \caption{Defining graph of a non-relatively hyperbolic RACG with locally-connected Morse boundary}
  \end{figure}

In proving Theorem~\ref{thm:intro_not_wide_avoidant} and Theorem~\ref{thm:intro_connected}, we drew much of our inspiration from the work of Mihalik-Ruane-Tschantz.
In the context of visual boundaries of the Davis--Moussang complex of right-angled Coxeter groups, 
Mihalik-Ruane-Tschantz give sufficient condition for their local connectivity \cite{Mihalik-Ruane-Tschantz}. 
In fact, their conditions on the defining graph for local connectivity of the visual boundary 
(no product separators and no virtual factor separators) is equivalent to wide-avoidant for right-angled Coxeter groups. 
Theorem~\ref{thm:intro_connected} generalizes their result to Morse boundaries of Coxeter groups. 

Mihalik-Ruane-Tschantz define the concept of a \emph{filter}, a certain type of planar van-Kampen diagram with boundary two geodesic rays in the Davis complex. 
We generalize their definition to general Coxeter groups and show that a filter is a stable subset connecting two Morse geodesics. We show that filters between two geodesic rays 
always exist given the hypothesis of Theorem~\ref{thm:intro_connected}. 

In order to prove Theorem~\ref{thm:intro_connected}, we prove a characterization of Morse geodesics in affine-free Coxeter group which is of independent interest:
\begin{introtheorem}\label{thm:intro_thm_morse_char}
Given a geodesic ray $\gamma$ in the Davis complex of an affine-free Coxeter group, the following are equivalent:
\begin{enumerate}
    \item The ray $\gamma$ is Morse,
    \item there exists a constant $K$ such that any subpath of length larger than $K$ does not have its label contained in a wide special subgraph, and
    \item the ray $\gamma$ spends uniformly bounded time in cosets of wide special subgroups.
\end{enumerate} 
\end{introtheorem}

Note that the equivalence between (2) and (3) is immediate. 
Although the above theorem is perhaps intuitive, its proof turned out to be quite technical. 
In its proof, we utilize a result of Caprace \cite{Caprace-conjugacy-2spherical} that allowed us to recognize affine subgroups from certain configuration of walls.
Our proof then required a careful analysis of CAT(0) bridges and disk diagrams in Coxeter groups.
We conjecture that Theorem~\ref{thm:intro_thm_morse_char} holds for all Coxeter groups, not just affine-free ones.

This article is organized as follows. Section~\ref{sec:background} gives necessary background on the Morse boundary and Coxeter groups, 
while Section~\ref{sec:wide-avoidant_def} gives formal definitions of wide-avoidant and wide-spherical-avoidant Coxeter groups.
Section~\ref{sec:disconnected} proves Theorem~\ref{thm:intro_not_wide_avoidant} regarding disconnected Morse boundaries. 
Next, in Section~\ref{sec:morse_geodesic_char} we prove Theorem~\ref{thm:intro_thm_morse_char} characterizing Morse geodesics in affine-free Coxeter groups. 
We then define filters in Section~\ref{sec:filters} and show that they can always be constructed under certain conditions. Finally, we prove 
Theorem~\ref{thm:intro_connected} in Section~\ref{sec:connected}.

\subsection*{Open questions}
\begin{enumerate}
    \item Can the affine-free condition in the statements of Theorem \ref{thm:intro_connected} and Theorem \ref{thm:intro_thm_morse_char} be dropped? \\
    The affine-free condition is used to prove Lemma \ref{lem:commuting_dual_curves}, where we use a theorem of Caprace that recognizes affine subgroups of Coxeter groups via a pattern of wall intersections. 
    This lemma in turn is used to prove a characterization of Morse geodesics by the length of time they spend in wide subgroups (Theorem \ref{thm:morse_char}), which we then use to show that the filters we construct in Section \ref{sec:filters} are stable.
    \item Can Theorem~\ref{thm:intro_connected} be extended to the visual boundary of the Davis complex? \label{question-visual-boundary}
    The Mihalik--Ruane--Tschantz proof of local connectivity for the visual boundary \cite{Mihalik-Ruane-Tschantz} uses combinatorial techniques for RACGs that do not generalize to all Coxeter groups, so a new idea will need to be found to transfer our results to the visual boundary.  
    \item Can one construct fans \`a la Champetier \cite{champetier} to study when the Morse boundaries of Coxeter groups are \emph{not} planar? 
    Interestingly, the existence of a non-planar graph is not a well-defined invariant of the visual boundary $\mathrm{CAT}(0)$ group \cite{schreve-stark}, but it is for the Morse boundary.
    \item As we now have a large class of Coxeter groups with connected and locally connected Morse boundaries, their Morse boundaries may have interesting \v{C}ech cohomology as developed in \cite{FKSZ}.
\end{enumerate}

\subsection*{Acknowledgements}
The authors would like to thank Pierre-Emmanuel Caprace, Annette Karrer, Kim Ruane, and Stefanie Zbinden for helpful conversations. 
The first author was supported in part by an SNSF Ambizione Fellowhsip. 

\section{Background} \label{sec:background}

In this section we give background on the Morse boundary and Coxeter groups.

\subsection{Morse boundary} \label{sec:background:morse boundary}

\begin{definition}
Let $N \colon [1,\infty) \times [0,\infty) \to [0,\infty)$ be a function which is non-decreasing and continuous in the second factor. The quasi-geodesic $\gamma \colon I \to X$ 
is an  \emph{$N$-Morse quasi-geodesic} if for any subinterval $[s,t] \subseteq I$, any $(k,c)$-quasi-geodesic with endpoints $\gamma(s)$ and $\gamma(t)$ is contained in the $N(k,c)$-neighborhood of the image of $\gamma\vert_{[s,t]}$.
The function $N$ is called a \emph{Morse gauge}. 
\end{definition}

As a set, the Morse boundary of a proper geodesic metric space $X$ is the set of all Morse geodesic rays in $X$ based at a point 
$o\in X$ and modulo asymptotic equivalence; given a geodesic ray $\gamma \colon [0,\infty)\to X$, 
we denote the equivalence class as $[\alpha]$. 
This collection is called the \emph{Morse boundary} of $X$ and is denoted by $\partial_{*} X$ or $\partial_{*} X_o$ 
if we wish to emphasize the basepoint $o$. 

To topologize the boundary, first fix a Morse gauge $N$ and consider the subset of the Morse boundary that consists of all rays in $X$ 
with Morse gauge at most $N$:  
\begin{equation*} \partial_{*}^N X_o= \{[\alpha] \mid \exists \beta \in [\alpha] 
    \text{ that is an $N$-Morse geodesic ray with base point }o\}. \end{equation*} 
One can topologize this set with the compact-open topology. 
This topology is equivalent to the one defined by a neighborhood basis: 
at each point $[\alpha] \in \partial_{*}^N X_o$, and every $n \in \mathbb{N}$, let $V_n^N( \alpha)$ 
denote the set of $N$-Morse geodesic rays $\gamma \colon [0,\infty) \to X$ with $\gamma(0)=o$ and $d(\alpha(t), \gamma(t))< \delta_N$ 
for all $t<n$, where $\delta_N$ is a constant that depends only on $N$. 
A neighborhood basis of $[\alpha]$ in $\partial_{*}^N X_o$
is the set $$\{V_n^N(\alpha) \mid n \in \mathbb{N} \}.$$
The constant $\delta_N$ is computed in \cite{Cordes2017, cordessistozbinden} and is chosen so that any ray in the 
equivalence class of $\alpha$ is included in $V_n^N(\alpha)$, i.e., so that the topology is well defined.

Let $\mathcal M$ be the set of all Morse gauges. 
We put a partial ordering on $\mathcal M$ so that for two Morse gauges $N, N' \in \mathcal M$, we have that $N \leq N'$ if and only if 
$N(k,c) \leq N'(k,c)$ for all $k \geq 1$ and $c \geq 0$. We define the Morse boundary of $X$ to be
 \begin{equation*} \partial_{*} X_o=\varinjlim_\mathcal{M} \partial^N_* X_o \end{equation*} with the induced direct limit topology, i.e., a set $U$ is open in $\partial_{*} X_o$ if and only if $U \cap \partial^N_* X_o$ is open for all $M \in \mathcal{M}$.  For more details on the Morse boundary see \cite{Cordes2017, cordessistozbinden}. 

The first salient feature of the Morse boundary is that it is invariant under quasi-isometry regardless of the geometry of the space. 

\begin{theorem}[\cite{Cordes2017}]\label{thm:Morse_boundar_QI_invariant}
Let $X$ and $Y$ be proper geodesic metric spaces. Every quasi-isometry $f\colon X \to Y$ induces a homeomorphism $\partial f \colon \partial_{*} X_o \to \partial_{*} Y_{f(o)}$.
\end{theorem}

As a corollary of this theorem, a change in basepoint induces a homeomorphism on the boundary. This justifies the omission of the basepoint from the notation. More generally, we will usually assume the basepoint is fixed and omit it from the notation $\partial^N_* X_o$ as well. The reader is warned, however, that unlike $\partial_* X$, the substrata, $\partial^N_* X$, \emph{do} depend on a choice of basepoint.

We prove a quick lemma about the concatenation of Morse geodesics we will use later.

\begin{lemma} \label{lem:concatenation of morse is morse}
    Let $X$ be a geodesic metric space. Let $[x, y]$ and $[y,z]$ be $N$-Morse geodesics. Further assume that the concatenation $\gamma =[x, y]\cup[y,z]$ 
    is also a geodesic. Then there is $M$, depending only on $N$, so that $\gamma$ is $M$-Morse. 
\end{lemma}

\begin{proof}
    First note that $[x,y]\cup[y,z] \cup \gamma$ is a geodesic triangle with two $N$-Morse sides $[x, y]$ and $[y,z]$. Now we apply Lemma 2.2 of \cite{Cordes2017} to conclude that $\gamma$ is $N'$-Morse where $N'$ depends on $N$.
\end{proof}

\subsection{Coxeter groups}

A \emph{Coxeter graph} is a simplicial graph $\Gamma$ whose edges are labeled by integers in $\{2,3, 4, \dots\}$. Given two vertices $s, t \in \Gamma$, we let $m_{st}$ denote the label between them. We let $V(\Gamma)$ denote the vertex set of $\Gamma$.
The \emph{Coxeter group $W_\Gamma$} associated to the Coxeter graph $\Gamma$ is the group with presentation
\[\left \langle V(\Gamma) ~|~ s^2 = 1 ~\text{ for all } s \in V(\Gamma),~(st)^{m_{st}} = 1 \text{ for each } (s,t) \in E(\Gamma) \right \rangle\] 
\begin{remark}
    Another common convention is that two vertices $s$ and $t$ in the Coxeter group commute in $W_\Gamma$ if they are non-adjacent. For our purposes, we do not use this description. Note that, in our convention, two non-adjacent vertices generate an infinite dihedral group. 
\end{remark}

Given a subgraph $\Delta \subset \Gamma$, we say that $\Delta$ is \emph{induced} if every edge in $\Gamma$ between vertices of $\Delta$ is also an edge of $\Delta$. 
The subgraph \emph{induced by $\Delta$} is the the induced subgraph of $\Gamma$ with vertex set $V(\Delta)$.
By the usual abuse of notation, we let $W_\Delta$ denote the subgroup of $W_\Gamma$ generated by the vertices of $\Delta$. 
It is well known that this subgroup is indeed isomorphic to $W_{\Delta'}$ where $\Delta'$ is the Coxeter graph induced by $\Delta$ (see \cite{Davis} for instance).

Let $W_\Gamma$ be a Coxeter group. 
We say that $W_\Gamma$ is \emph{irreducible} if it does not split as a product of groups and $W_\Gamma$ is \emph{affine} if it is irreducible and a Euclidean reflection group. 
The classification of affine Coxeter groups is well-known 
(see \cite{Davis}).
The \emph{rank} of a Coxeter group $W_\Gamma$ is the number of vertices in $\Gamma$.
We say that $W_\Gamma$ is \emph{affine-free} if for any subgraph $\Delta \subset \Gamma$, $W_\Delta$ is not an affine Coxeter group of rank $3$ or larger. 
Note, this is equivalent to saying that $W_\Gamma$ has no irreducible affine reflection subgroups of rank $3$ or larger \cite{Caprace-conjugacy-2spherical}.

Let $w = s_1 \dots s_n$, with each $s_i \in V(\Gamma)$. We call $w$ a \emph{word} $W_\Gamma$. 
The word $w'$ is an \emph{expression} for the word $w$ if these words are equal as elements of $W_\Gamma$. The word $w$ is \emph{geodesic} if it is a minimal length expression for the corresponding element of $W_\Gamma$, i.e., $n$ is minimal over all possible expressions for $w$. We let $\text{Label}(w)$ denote the set of vertices in $t \in V(\Gamma)$ such that $t = s_i$ for some $1 \le i \le n$.  An element $w \in W_\Gamma$ is called a \emph{reflection} if it is conjugate to an element of $V(\Gamma)$. 

The Davis complex $\Sigma$ associated to $W_\Gamma$ is a piecewise Euclidean $\mathrm{CAT}(0)$ cell complex whose $1$-skeleton is the Cayley graph of $W_\Gamma$ respect to the generating set $V(\Gamma)$. The action of $W_\Gamma$ on its Cayley graph induces a properly discontinuous, cocompact action on $\Sigma$. By definition, a \emph{wall} of $\Sigma$ is the fixed point set of a reflection of $W_\Gamma$. Given a wall $H$ associated to a reflection $gsg^{-1}$ with $s \in V(\Gamma)$ and $g \in W_\Gamma$, we say that $H$ has \emph{type} $s$. Note that in Coxeter groups, no two distinct elements of $S$ are conjugate in $W$, so this is a well-defined notion. For more on the Davis complex see \cite[Chapters 7 and 12]{Davis}.

By a \emph{path} in $\Sigma$, we shall mean a path in the $1$-skeleton of $\Sigma$.
We let $\overline{\gamma}$ denote the word obtained by reading the labels of edges of $\gamma$ in order.
A path $\gamma$ in $\Sigma$ is \emph{geodesic} if it is a geodesic in the Cayley graph of $W_\Gamma$.
Equivalently, $\gamma$ is geodesic if and only if every wall in $\Sigma$ intersects $\gamma$ at most once if and only if $\bar{\gamma}$ is geodesic.

When we use the $\mathrm{CAT}(0)$ metric on $\Sigma$, we will call a geodesic path a \emph{CAT(0) geodesic path}. Note that CAT(0) geodesics are not necessarily contained in the $1$-skeleton of $\Sigma$.

\subsection{van-Kampen diagrams over Coxeter groups}
We refer to \cite[1.4]{Bahls} for additoinal background on van-Kampen diagrams over Coxeter groups.
Let ${D}$ be a van-Kampen diagram over a Coxeter group $W_\Gamma$. There is a label-preserving combinatorial map from ${D}$ to $\Sigma_\Gamma$. Under this map, edges and vertices of ${D}$ go to edges and vertices of the Cayley graph (thought of as the $1$-skeleton of $\Sigma$). We may assume that each $2$-cell in ${D}$ is an even length polygon, since every relation is of even length and any $2$-cell with boundary with boundary label of the form $s^2$ maps to an edge of $\Sigma$ and can therefore be collapsed to an edge in ${D}$. 

Recall that each $2n$-gon in $\Sigma$ is intersected by $n$-walls which all intersect in the center of $2n$-gon. Using the projection of ${D}$ to $\Sigma$, we can pull back these $n$ intersections to ${D}$ as $n$ line segments which we refer to as wall segments. For each cell of $D$, we can draw in its wall segments. Note that each edge of $D$ is either a boundary edge or is contained in two polygons. In the latter case, there are exactly two wall segments whose intersection as contained in this edge. We say that these two wall segments are adjacent. Note that, starting at any wall segment in a polygon of ${D}$, one can extend this segment through adjacent wall segments. This can be done in both directions. A maximal such collection of segments is called a \emph{dual curve of ${D}$}. Note that the projection of a dual curve to $\Sigma$ is contained in a unique wall. 

\section{Wide-avoidant and wide-spherical-avoidant Coxeter groups} \label{sec:wide-avoidant_def} 

In this section we set up some definitions which are crucial for the rest of the paper. The first such notion is that of a wide group which is a broad definition \cite{Behrstock-Drutu-Mosher}. We define it here in the setting of special subgroups of Coxeter groups (see \cite{Behrstock-Hagen-Sisto}).

\begin{definition} \label{def:wide_subgroups}
Let $W_\Gamma$ be a Coxeter group.
We call an induced subgraph $\Delta \subset \Gamma$ \emph{wide} if there is a decomposition $V(\Delta) = P \sqcup Q$ (with $Q$ possibly empty) such that there is an edge labeled by $2$ between  any vertex in $P$ and any vertex in $Q$, and additionally either
\begin{enumerate}
    \item  both $W_{P}$ and $W_{Q}$ are infinite, or
    \item $W_P$ is irreducible, affine of rank at least $3$. 
\end{enumerate}
 We call $(P,Q)$ a \emph{wide decomposition} of $\Delta$. Note that, when $Q$ is not empty, $W_\Gamma \cong W_P \times W_Q$. We also say that $W_\Gamma$ is wide if $\Gamma$ is wide.
\end{definition}

Wide subgraphs are of importance to us, as they exactly correspond to the special subgroups with empty Morse boundary:

\begin{lemma} \label{lem:non_empty_bdry}
 An infinite Coxeter group has empty Morse boundary if and only if it is wide.
\end{lemma}
\begin{proof}
    If $W_\Gamma$ is not wide, then $W_\Gamma$ contains a rank 1 element by \cite[Proposition 4.5]{Caprace-Fujiwara}. As rank-1 elements in $\mathrm{CAT}(0)$ spaces are Morse (see \cite[Proposition~3.3]{Kapovich-Leeb} and \cite{Charney-Sultan}), the Morse boundary of $W_\Gamma$ is non-empty.
    On the other hand, if $W_\Gamma$ is wide, then it has no Morse geodesics by \cite{Drutu-Mozes-Sapir}. This can also readily be seen as then $W_\Gamma$ either splits as a product of infinite groups or is virtually affine of rank at least $3$ (and so, up to a finite-index subgroup, acts geometrically on $\mathbb{E}^n$
    for some $n \ge 2$).
\end{proof}

Next, we define the notion of wide-avoidant graphs/groups. We later show in Theorem~\ref{thm:not_prod_conn} that Coxeter groups that are not wide-avoidant have either empty or disconnected Morse boundary.
\begin{definition}[wide-avoidant] \label{def:wide-avoidant}
	We say that the Coxeter graph $\Gamma$ and the Coxeter group $W_\Gamma$ are \emph{wide-avoidant} if given any wide subgraph $\Delta$ and any pair of vertices $s, t \in \Gamma$, then there exists a path $\alpha \subset \Gamma$ from $s$ to $t$ such that $\alpha \cap \Delta \subset \{s, t\}$.
\end{definition}

We next need the notion of wide-spherical-avoidant graphs/groups. Essentially, the definition is similar to wide-avoidant, but paths needs to avoid both a wide and spherical subgraph. Moreover, the actual definition of wide-spherical-avoidant is stronger than this (See Definition~\ref{def:wide-spherical-avoidant}(3)), allowing us to eventually prove Morse boundary connectivity properties of a larger set of Coxeter groups (see Theorem~\ref{thm:MB is connected}).

\begin{definition}[Special Join]
    Given a Coxeter graph $\Gamma$, a \emph{special join} is a triplet $(P, Q, K)$ where $P$, $Q$ and $K$ are induced, disjoint subgraphs of $\Gamma$ such that:
    \begin{enumerate}
        \item $(P,Q)$ is a wide decomposition of the induced subgraph spanned by $P \cup Q$.
        \item $K$ is either empty or $W_K$ is finite
        \item For all $k \in K$ and $s \in P$, there is an edge (with any label) in $\Gamma$ from $k$ to $s$.
    \end{enumerate}
\end{definition}

\begin{remark} \label{rmk:special_joins_in_racgs}
    If $W_\Gamma$ is a \emph{right-angled} Coxeter group, then special joins correspond to wide subgroups and vice versa. 
    For let $(P, Q, K)$ be a special join in $\Gamma$. Then $W_{P \cup Q \cup K} = W_P \times W_{Q \cup K}$
    is a wide subgroup. Conversely, suppose that $W_{\Delta}$ is a wide subgroup of a right-angled Coxeter group.
    Then $\Delta$ has a wide decomposition $(P, Q)$, and $(P, Q, \emptyset)$ is a special join. 
\end{remark}

\begin{definition}[wide-spherical-avoidant] \label{def:wide-spherical-avoidant}
	We say that the Coxeter graph $\Gamma$ and the Coxeter group $W_\Gamma$ are \emph{wide-spherical-avoidant} if given any special join $(P, Q, K)$ and any pair of vertices $s, t \in \Gamma \setminus K$, then there exists a path $\alpha$ in $\Gamma$ from $s$ to $t$ such that $\alpha \cap (K \cup P \cup Q) \subset \{s, t\}$.
\end{definition}

Throughout the remainder of the paper, we will need the following important constants associated to a Coxeter group.

\begin{definition} \label{def:constants}
    Let $W_\Gamma$ be a Coxeter group. 
    \begin{itemize}
        \item Define $V_\Gamma$ be the number of vertices of $\Gamma$.
        \item Define $M_\Gamma$ be the smallest constant such that $W_{\text{Label}(w)}$ is infinite
        for any geodesic word $w$ with $|w| > M_\Gamma$.
        \item Define $R_{\Gamma} = \max\{ m_{st}~|~(s,t) \in E(\Gamma)\}$ (where $m_{st}$ is the label of the edge $(s,t)$).
    \end{itemize}
\end{definition}

    Note that $M_\Gamma$ is simply the maximal length of a word contained in a finite special subgroup of $W_\Gamma$. 

\section{Disconnected Morse boundaries} \label{sec:disconnected}

The main result of this section, Theorem~\ref{thm:not_prod_conn}, 
shows that if a Coxeter groups is not wide-avoidant, then it either has empty or disconnected Morse boundary.  

The next lemma gives necessary conditions for the Morse boundary of a Coxeter group to be empty or disconnected. 
To prove the lemma we will use the relative Morse boundary introduced by Fioravanti and Karrer in \cite{Fioravanti-Karrer}. 
In short,  the \emph{relative Morse boundary} of a subgroup $H$ in $G$, denoted by $(\partial_*H, G)$ is the subspace of the Morse boundary $\partial_* G$ of $G$ that consists of all equivalence classes of Morse geodesic rays in $G$ which are at bounded distance from~$H$.

\begin{lemma} \label{lem:bdry_of_splitting}
    Let $G = G_1 \star_H G_2$ be a non-trivial splitting 
    such that $H$ is finitely generated, undistorted and has empty relative Morse boundary. 
    Then, the Morse boundary of $G$ is empty or disconnected. 
\end{lemma}

\begin{proof}
    Assume that the Morse boundary of $G$ is non-empty. 
    If the relative Morse boundaries $(\partial_* G_1, G)$ and $(\partial_* G_2,G)$ are both empty, then $\partial_* G$ is totally disconnected by Corollary~B of \cite{Fioravanti-Karrer} and thus disconnected. 
    Thus, without loss of generality, we assume that $(\partial_* G_1, G)$ is non-empty. Then we know any coset $gG_1$ of $G_1$ has a Morse geodesic. By Remark~3.6 of \cite{Fioravanti-Karrer}, we are in two disjoint closed subsets of $\partial_* G$ and thus $\partial_* G$ is disconnected. 
\end{proof}

The proof of the following lemma is similar to that of \cite[Lemma 5.3]{Mihalik-Ruane-Tschantz}.
\begin{lemma}\label{lem:not_prod_conn}
    Let $W_\Gamma$ be a Coxeter group that is not wide-avoidant. Then one of the following holds:
    \begin{enumerate}
        \item $W_\Gamma$ is wide,
        \item $W_\Gamma$ splits over a finite group or 
        \item There is a non-trivial decomposition $\Gamma = \Gamma_1 \cup \Gamma_2$ such that $W_\Gamma = W_{\Gamma_1} *_{W_\Delta} W_{\Gamma_2}$ where $\Delta = \Gamma_1 \cap \Gamma_2$, $W_{\Delta}$ infinite and $\Delta$ is contained in a wide subgraph.
    \end{enumerate}
\end{lemma} 
\begin{proof}
    First note that $W_\Gamma$ is infinite, as, by definition, finite Coxeter groups are wide-avoidant. 
    We suppose that (1) and (2) do not hold and show that (3) holds.
    As $\Gamma$ is not wide-avoidant, there exists a wide subgraph $\Pi$, 
    and vertices $s, t \in \Gamma$, such that if $\gamma \subset \Gamma$ 
    is a path from $s$ to $t$, then 
    $\gamma \cap \Pi \not\subset \{s, t\}$.
    
    There are a few cases to check, depending on the number of components of $\Gamma \setminus \Pi$.
    First note that $\Gamma \setminus \Pi$ must have at least one component because are assuming that $W_\Gamma$ is not wide.
    
    Suppose first that $\Gamma \setminus \Pi$ contains more than one component. Let $C$ be one such component.
    We set $\Gamma_1$ to be the graph induced by $C \cup \Pi$, and we set $\Gamma_2$ to be the graph induced by $\Gamma \setminus C$.
    Consequently, $\Gamma_1 \cap \Gamma_2 = \Pi$, and $W_\Gamma = W_{\Gamma_1} *_{W_\Pi} W_{\Gamma_2}$.
    Thus, (3) follows by setting $\Delta = \Pi$ and noting that wide groups are infinite.
    
    Suppose now that $\Gamma \setminus \Pi$ contains exactly one component. 
    We first show that either $\text{star}(s) \subset \Pi$ or $\text{star}(t) \subset \Pi$.
    Suppose, for a contradiction, that exists some $s' \in \text{star}(s) \setminus \Pi$ and $t' \in \text{star}(s) \setminus \Pi$.
    As $\Gamma \setminus \Pi$ contains exactly one component, there is a path $\gamma'$ in $\Gamma \setminus \Pi$ connecting $s'$ to $t'$.
    However, we now get a path $\gamma$ connecting $s$ to $t$ by concatenating the edge $(s, s')$, $\gamma'$ and the edge $(t, t')$.
    Moreover, $\gamma \cap \Pi = \{s, t\}$. 
    This is a contradiction.
	
	Thus, up to relabelling $s$ and $t$, we have that $\text{star}(s) \subset \Pi$.
	Let $\Delta$, $\Gamma_1$ and $\Gamma_2$ be the subgraphs induced respectively by $\text{link}(s)$, $\text{star}(s)$ and $\Gamma \setminus \{s\}$.
	It is now readily checked that (3) holds with these choices.
\end{proof}

\begin{theorem} \label{thm:not_prod_conn}
    Let $W_\Gamma$ be a Coxeter group that is not wide-avoidant. Then one of the following holds 
	\begin{enumerate}
		\item $W_\Gamma$ has empty Morse boundary and is wide, 
        \item $W_\Gamma$ has disconnected Morse boundary and either:
            \begin{enumerate}
                \item splits over a finite group, or 
                \item has a non-trivial decomposition $\Gamma = \Gamma_1 \cup \Gamma_2$ such that 
                $W_\Gamma = W_{\Gamma_1} *_{W_\Delta} W_{\Gamma_2}$ where $\Delta = \Gamma_1 \cap \Gamma_2$, and $\Delta$ is contained in a wide subgraph.
            \end{enumerate} 
	\end{enumerate}
\end{theorem}
\begin{proof}
        Note that $W_\Gamma$ must be infinite as otherwise it is wide-avoidant.
	If $W_\Gamma$ is wide, then the Morse boundary of $W_\Gamma$ is clearly empty.
	Thus, we assume that $W_\Gamma$ is not wide and show that (2) holds. 
	
	If $W_\Gamma$ has a non-trivial splitting over a finite group, then its Morse boundary is non-empty (any periodic geodesic that goes through infinitely many cosets is Morse) and is not connected by Lemma~\ref{lem:bdry_of_splitting}.
	Thus, (2) holds in this case.
	We can then further assume that $W_\Gamma$ does not split over a finite group.
	
	Note that the Morse boundary of $W_\Gamma$ is non-empty by Lemma~\ref{lem:non_empty_bdry}.
	By Lemma~\ref{lem:not_prod_conn} there is a decomposition $\Gamma = \Gamma_1 \cup \Gamma_2$ such that 
    $W_\Gamma = W_{\Gamma_1} *_{W_\Delta} W_{\Gamma_2}$ where $\Delta = \Gamma_1 \cap \Gamma_2$, $W_{\Delta}$ infinite and $\Delta$ 
    contained in a wide subgraph of $\Gamma$.
	Consequently, the Morse boundary of $W_\Gamma$ is not connected by Lemma~\ref{lem:bdry_of_splitting}.
\end{proof}

\section{Characterization of Morse geodesics \\in affine-free Coxeter groups} \label{sec:morse_geodesic_char}

The main result of this section is the following characterization of Morse geodesics.

\begin{theorem} \label{thm:morse_char}
	Let $W_\Gamma$ be an affine-free Coxeter group. A geodesic ray $\gamma$ in the Davis complex $\Sigma = \Sigma_\Gamma$ is Morse if and only if there exists a constant $K$ such that any subpath $\gamma' \subset \gamma$ of length larger than $K$ does not have $\text{Label}(\gamma')$ contained in a wide subgraph of $\Gamma$.
\end{theorem}

The following lemma gives equivalent conditions for $\text{Label}(\gamma)$ of a path $\gamma$ to be contained in a wide subgraph. 
\begin{lemma} \label{lem:in_wide}
    Let $\gamma$ be a path in the Davis complex of a Coxeter group $W_\Gamma$. Then the following are equivalent:
    \begin{enumerate}
     \item $\text{Label}(\gamma)$ is contained in a wide subgraph,
     \item $\gamma$ is contained in the coset of a wide subgroup, and
     \item The reflections associated to the walls dual to $\gamma$ are contained in a conjugate of a wide subgroup. 
    \end{enumerate}
\end{lemma}
\begin{proof}    
    Without loss of generality, we may assume that $\gamma$ is based at the identity vertex. 
    Let $e_1, e_2, \dots e_n$ be the edges of $\gamma$ in the order that they appear, and let $s_i$ be the label of $e_i$.
    Note that $r_i = s_1 \dots s_{i-1} s_i s_{i-1} \dots s_1$ is the reflection associated to the wall dual to $e_i$. 

    The equivalence between (1) and (2) is immediate. We now show the equivalence of (1) and (3).
    We first assume (1): $\text{Label}(\gamma)$ is contained in a wide subgraph $\Delta$. Consequently, $s_i \in \Delta$ for each $i$.
    Thus, these reflections are all contained in $W_\Delta$ and (3) follows.
    
    Conversely, suppose that the reflections associated to the walls dual to $\gamma$ are contained in a conjugate of a wide subgroup $W_\Delta$.
    As $\gamma$ is based at the identity, we may assume that this conjugate is $W_\Delta$ itself.
    Now, note that $s_i = s_{i-1} \dots s_1 r_i s_1 \dots s_{i-1}$. Thus, $s_i \in W_\Delta$ for each $i$ and (1) follows. 
\end{proof}

The remainder of this section will build towards a proof of Theorem~\ref{thm:morse_char}.

\subsection{Recognizing affine and wide subgroups}

The proof of Theorem~\ref{thm:morse_char}, requires us to recognize affine subgroups by analyzing patterns of walls in the Davis complex. 
Caprace's theorem, stated below, allows us to do just this. We also prove Lemma~\ref{lem:wall_intersection_containment} which facilitates applications of Caprace's theorem.

\begin{definition}
A set of walls $H_0, \dots, H_m$ in a Davis complex forms a \emph{pencil} if they are pairwise non-intersecting and $H_j$ separates $H_i$ and $H_k$ for all $0 \le i < j < k \le m$.
\end{definition}

\begin{theorem}[\cite{Caprace-conjugacy-2spherical}, Theorem 8] \label{thm:recognizing_triangle_group}
	Let $W_\Gamma$ be a Coxeter group. There exists an integer $C$, depending only on $\Gamma$, such that the following holds. 
	Let $K$, $K'$, $L_0, \dots, L_C$ be distinct walls in the Davis complex of $W_\Gamma$ that satisfy:
	\begin{enumerate}
		\item $\emptyset \neq K \cap K' \subset L_0$,
		\item $L_0, \dots, L_C$ forms a pencil, and
		\item for each $1 \le i \le C$, $L_i$ intersects both $K$ and $K'$.
	\end{enumerate}
	Then, the subgroup generated by reflections corresponding to the walls $K$, $K'$, $L_0, \dots, L_C$ is isomorphic to a Euclidean triangle group. Moreover, it is contained in a conjugate of an affine, irreducible, special subgroup of rank at least~$3$.
\end{theorem} 

Given a Coxeter group $W_\Gamma$, we call the smallest integer $\Lambda_{\Gamma}$ given by Theorem~\ref{thm:recognizing_triangle_group}, its \emph{ladder constant}.
We call a set of walls $K$, $K'$, $L_0, \dots, L_C$ satisfying the condition of Theorem~\ref{thm:recognizing_triangle_group} a \emph{ladder} in $W_\Gamma$.

The next lemma, guarantees that condition (1) of Theorem~\ref{thm:recognizing_triangle_group} is automatically satisfied whenever $K$ intersects $L_0$, and $K' = r(K)$ where $r$ is the reflection corresponding to $L_0$.
\begin{lemma} \label{lem:wall_intersection_containment}
     Let $K$ and $L$ be intersecting walls in a Davis complex $\Sigma$ whose corresponding reflections do not commute.
     Then $K \cap r(K) \subset L$ where $r$ is the reflection associated to $L$.
\end{lemma}

Before proving this lemma, we will need the following:

\begin{lemma} \label{lem:intersecting_walls}
     Let $c$ be a $2$-cell of a Davis complex $\Sigma$, let $H$ be a wall intersecting $c$, and let $\pi: \Sigma \to c$ be the CAT(0) closest point projection. Then for any $x \in \Sigma$, $x \in H$ if and only if $\pi(x) \in H \cap c$. 
\end{lemma}
\begin{proof}
    To see one direction of the claim let $x \in H$ and suppose, for a contradiction, that $\pi(x) \notin H \cap c$. 
    Let $z$ be the reflection associated to $H$. The isometry $z$ acts on $c$ as a reflection about the line segment $H \cap c$. Thus, $z(\pi(y)) \neq \pi(z)$. As $x \in H$, $z(x) = x$. 
    Thus, there are two minimal length geodesics, $[x, \pi(x)]$ and $[x, z(\pi(x))]$, from $x$ to $c$. However, this is impossible as $c$ is a covex subset of a CAT(0) space.
    Hence, it must be indeed the case that $\pi(x) \in H \cap c$. 
    
    To see the other direction of the claim, let $x \in \Sigma$ such that $\pi(x) \in H \cap c$. 
    Suppose, for a contradiction, that $x \notin H$.
    Consider the minimal length geodesic $\gamma$ from $x$ to $H$. 
    As the endpoints of $\gamma$ project to $H \cap c$, we have that $\pi(\gamma) \subset H \cap c$.
    By the structure of the Davis complex there is an $\epsilon > 0$ and a neighborhood $P$ of $H$ that is isomorphic to $H \times [-\epsilon, \epsilon ]$ with $H$ identified with $H \times \{0\}$.
    As $x \notin H$,  there is a point $y \in (\gamma \cap P) \setminus H$.
    However, for all $p \in P$, $\pi(p) \in c \cap H$ if and only if $p \in H$.
    But, $y \notin H$ and $\pi(y) \in c \cap H$.
    This is a contradiction, and the lemma follows.
\end{proof}

\begin{proof}[Proof of \ref{lem:wall_intersection_containment}]
    As $L \cap K \neq \emptyset$, by the structure of $\Sigma$, there is a Euclidean $2$-cell $c$ such that $L \cap K \cap c \neq \emptyset$. 
    As the reflections corresponding to $L$ and $K$ do not commute, $L \cap K \cap r(K)$ is a single point $p$ (i.e., the center of the polygon $c$).
    Define $Y := K \cap r(K)$. 
    By Lemma~\ref{lem:intersecting_walls}, $\pi(K) = K \cap c$ and $\pi(r(K)) = r(K) \cap c$ where $\pi$ is the CAT(0) closest point projection to $c$. Thus, $\pi(Y) = p$.
    As $p \in L$, we conclude again from Lemma~\ref{lem:intersecting_walls} that $y \in L$ for all $y \in Y$.
\end{proof}

\subsection{Recognizing wide subgroups}
In this subsection, we prove Proposition~\ref{prop:contained_in_wide} showing that certain reflection subgroups must be contained in wide subgroups. A main tool in the proof of Proposition~\ref{prop:contained_in_wide} is an important result of Deodhar completely describes the normalizer of an infinite, irreducible special subgroup (see Theorem~\ref{thm:deodhar_normalizer_of_parabolic_subgrp} below). We also refer the reader to \cite[Section 2.2.1]{Bahls} for a description of Deodhar's result.

Given a subset $J$ of the vertices of a Coxeter graph $\Gamma$, we define $J^{\perp}$ to be the set of all vertices of $V(\Gamma) \setminus J$ that commute with all vertices of $J$.
Given a subgroup $G < W_\Gamma$, we denote its normalizer and centralizer (in $W_\Gamma$) by $\mathcal{N}(G)$ and $\mathcal{Z}(G)$ respectively.

\begin{theorem}[\cite{Deodhar82}] \label{thm:deodhar_normalizer_of_parabolic_subgrp}
Let $W_J<W_\Gamma$ be an infinite, irreducible, special subgroup of a Coxeter group $W_\Gamma$. Then 
\[\mathcal{N}(W_J) = W_{J \cup J^{\perp}} = W_J \times \mathcal{Z}(W_J) \]     
\end{theorem} \label{lem:deodhar_normalizer_of_parabolic_subgrp}

Recall, a reflection in a Coxeter group $W_{\Gamma}$ is a conjugate of a standard generator (i.e. a vertex of $V(\Gamma)$).
A \emph{reflection subgroup} of $W_\Gamma$, is a subgroup $G < W_\Gamma$ that is generated by reflections. 
A \emph{parabolic subgroup} of $W_\Gamma$ is one which is conjugate to a special subgroup. 
Given a subgroup $G < W_\Gamma$, the \emph{parabolic closure} $\text{Pc}(G)$ of $G$ is defined as the intersection of all parabolic subgroups containing $G$. Note that $\text{Pc}(G)$ is itself a parabolic subgroup, as the intersection of two parabolic subgroups is a parabolic subgroup (see \cite[Lemma 5.3.6]{Davis} for instance).

\begin{proposition} \label{prop:contained_in_wide}
    Let $G_1$ and $G_2$ be finitely generated, infinite reflection subgroups of a Coxeter group $W_\Gamma$ such that $[G_1,G_2] = \{1\}$. Then $G = \langle G_1 \cup G_2 \rangle$ is contained in a wide subgroup of $W_\Gamma$.
\end{proposition}
\begin{proof}
    By \cite{Deodhar-reflection_subgroups}, there is a finite set of reflection $R_1 \subset G_1$ (resp. $R_2 \subset G_2$) such that $(G_1, R_1)$ (resp. $(G_2, R_2)$) is a Coxeter system.
    As each $G_i$ is infinite, we may write
    \[G_1 = H_1 \times H_1' \text{ and } G_2 = H_2 \times H_2'\]
    where $H_1$ (resp. $H_2$) is a maximal special subgroup that is irreducible and infinite in the Coxeter system $(G_1, R_1)$ (resp. $(G_2, R_2))$, and $H_1'$ (resp. $H_2'$) is the subgroup generated by $G_1 \setminus H_1$ (resp. $G_2 \setminus H_2$) or the empty set if $G_1 \setminus H_1 = \emptyset$ (resp. $G_2 \setminus H_2 = \emptyset$). 
    
    Let $gW_Jg^{-1} := \text{Pc}(H_1)$ where $g \in W_\Gamma$ and $W_J$ is a special subgroup.
    As $H_1$ is infinite, so is $gW_Jg^{-1}$. Additionally, $gW_Jg^{-1}$ is irreducible by \cite[Lemma 2.1]{Caprace-relhyp}.
    Set $Z := \mathcal{Z}(gW_Jg^{-1})$.
    By Theorem~\ref{thm:deodhar_normalizer_of_parabolic_subgrp} and as $\mathcal{N}(gW_Jg^{-1}) = g \mathcal{N}(W_J) g^{-1}$, we have that $Z = g (W_{J^{\perp}}) g^{-1}$. 
    
    By \cite[Lemma 2.1]{Caprace-relhyp}, either 
    \begin{enumerate}
        \item $\text{Pc}(H_1) = \text{Pc}(H_2)= \text{Pc}( \langle H_1 \cup H_2 \rangle)$ is irreducible, affine of rank at least $3$, or
        \item $\text{Pc}( \langle H_1 \cup H_2 \rangle) \cong \text{Pc}(H_1) \times \text{Pc}(H_2)$
    \end{enumerate}
    In either case, $Y := gW_{J \cup J^\perp}g^{-1}$  is a wide subgroup containing $\text{Pc}( \langle H_1 \cup H_2 \rangle)$. In the first case, this is clear. In the second, $\text{Pc}(H_2) \subset gW_{J^{\perp}} g^{-1}$ and so $g W_{J^{\perp}} g^{-1}$ is infinite and consequently $Y$ is wide.
    
    Let $h \in H_1' \cup H_2'$. We now show that $h \in Y$, completing the proof of the lemma.  
    As $h$ commutes with $H_1$, $h (gW_Jg^{-1}) h^{-1}$ contains $H_1$. 
    By the definition of parabolic closure, $gW_Jg^{-1}$ is the smallest parabolic subgroup containing $H_1$.
    Moreover, the intersection of two parabolic subgroups is a parabolic subgroup.
    Consequently, $h (gW_Jg^{-1}) h^{-1} = gW_Jg^{-1}$ and so $h \in \mathcal{N}(gW_Jg^{-1})$. 
    By Theorem~\ref{thm:deodhar_normalizer_of_parabolic_subgrp}, $h \in gW_{J^\perp}g^{-1} \subset Y$.   
\end{proof}

\subsection{Bridges}
Given two convex subspaces of a CAT(0) space, their \emph{bridge} is the set of all minimal length geodesics between them. A bridge is convex and splits as a product $Y \times [0,I]$ where $I$ is the distance between the two convex subspaces (see \cite[II.2]{BH}).

\begin{definition}[Large bridge]
 We say that a bridge $B \cong Y \times [0,I]$ between a pair of walls $H$ and $K$ in the Davis complex $\Sigma_{\Gamma}$ is \emph{large} if $Y$ is unbounded and there exist $\Lambda_{\Gamma}$ distinct, pairwise non-intersecting walls that separate $H$ and $K$, where $\Lambda_{\Gamma}$ is the ladder constant.
 \end{definition}
 
 The main result of this subsection is Proposition~\ref{prop:large_bridges} stating that there is a natural wide subgroup associated to large bridges between walls in an affine-free Coxeter group.

\begin{proposition}\label{prop:large_bridges}
	Let $W_\Gamma$ be a
        Coxeter group.
	Let $H$ and $K$ be walls in the Davis complex $\Sigma = \Sigma_\Gamma$ such that the bridge between them is large.
	Let $\beta$ be a minimal length geodesic from $H$ to $K$. Then the reflections associated with the walls that intersect $\beta$ transversely are contained in a common wide subgroup.
\end{proposition}
\begin{proof}
Let $B \cong Y \times [0,I]$ be the bridge between $H$ and $K$.
    Let $(y, 0) \in Y \times [0, I]$ be the startpoint of $\beta$ and $(y,I)$ its endpoint.
    As $Y$ is convex and locally compact, by a compactness argument $Y$ contains an infinite CAT(0) geodesic ray $\gamma$ based at~$y$.

	Let $\lambda  = \gamma \times \{0\} \subset H$ and $\rho = \gamma \times \{I\} \subset K$. 
	The convex hull of $\lambda \cup \rho$ is a convex subset $F \subset \Sigma$ isometric to $[0,I] \times [0, \infty)$ (see \cite[II.2]{BH}). 
	In particular, $F$ is isometric to a convex subset of a $2$-dimensional flat.

	As $B$ is large, there is a pencil $H = H_0, \dots, H_{m+1} = K$ of walls with $m = \Lambda_{\Gamma}$ (the ladder constant).
	Let $\ell_i = F \cap H_i$, and note that $\ell_i$ is a geodesic in $F$. 
	As $H_0, \dots, H_{m+1}$ forms a pencil, the lines $\ell_0, \dots, \ell_{m+1}$ are parallel in $F$. Consequently, each $\ell_i$ forms a right angle with $\beta \subset F$.
	
    Let $\mathcal L$, $\mathcal R$ and $\mathcal Y$ be the set of walls in $\Sigma$ that intersect respectively $\lambda$, $\rho$ and $\beta$ transversely. Let $h_0, \dots, h_{m+1}$ be the reflections associated with $H_0, \dots, H_{m+1}$.

    We first suppose that $(\mathcal L \cup \mathcal R) \cap \mathcal Y \neq \emptyset$ and show the proposition follows. Let $X$ be in $(\mathcal L \cup \mathcal R) \cap \mathcal Y$, and let $r$ be its associated reflection. By symmetry, we can assume without loss of generality that $X \in \mathcal L \cap \mathcal Y$. Let $q = X \cap F$. Note that $h_0$ and $x$ do not commute as $q$ and $\lambda = H_0 \cap F$ do not form a right angle.
    Thus, the wall $h_0X$ is not equal to $X$.
    Moreover, $h_0X \in \mathcal L$ since $X \cap H_0 \neq \emptyset$. 
    Let $q' := h_0X \cap F$. Note that $q'$ intersects both $\lambda$ and $\rho$. Consequently, $h_0X$ intersects both $H_0$ and $H_{m+1}$.
    Consider now the wall $h_{m+1}h_0X$ and note that, by a similar argument, $h_{m+1}h_0X$ also intersects both $H_0$ and $H_{m+1}$.
    
    We have that $h_{m+1}h_0X$, $h_0X$, $H_0, \dots, H_{m+1}$ forms a ladder. By Theorem~\ref{thm:recognizing_triangle_group} and Lemma~\ref{lem:wall_intersection_containment}, $h_{m+1}h_0x$, $h_0x$, $h_0, \dots, h_{m+1}$ are all in common Euclidean triangle group that is a subgroup of a conjugate $G$ of an irreducible, affine, special reflection group of rank at least $3$. As $x = h_0(h_0x)$, $x \in G$ as well.

    Given any wall in $(\mathcal L \cup \mathcal R) \cap \mathcal Y$, its reflection is contained in an Euclidean triangle group which contains $h_0, \dots, h_{m+1}$. By \cite[Lemma~3.3(3)]{Caprace-Haglund}, however, this Euclidean triangle group is also a subgroup of $G$. This shows that the reflection associated to any wall in $(\mathcal L \cup \mathcal R) \cap \mathcal Y$ is in $G$. 

    We are left to show that reflections associated to walls in $\mathcal Y \setminus (\mathcal L \cup \mathcal R)$ are in $G$. Let $Z$ be such a wall, and let $z$ be its associated reflection. We have that $Z$ and $h_0X$ intersect at a non-right angle. Consider the wall $h_0xZ$. This wall intersects either $H_0$ or $H_{m+1}$ at a non-right angle.  By an argument similar to that of before, $gh_0xZ \in (\mathcal L \cup \mathcal R) \cap \mathcal Y$  where $g$ is in the group $\langle h_0, h_{m+1} \rangle$. Thus, $z \in G$. The proposition now follows for the case when $(\mathcal L \cup \mathcal R) \cap \mathcal Y \neq \emptyset$.

    We now assume that $(\mathcal L \cup \mathcal R) \cap \mathcal Y = \emptyset$. 
    Given any $X \in \mathcal L$ and any $Y \in \mathcal Y$, then $X \cap F$ intersects $Y \cap F$ at a right angle.
    Thus, the reflections $L$ associated to the walls in $\mathcal L$ all commute with the reflections $R$ associated to the walls in $\mathcal Y$.
    As $H_0$ and $H_{m+1}$ do not intersect, the subgroup generated by the reflections associated to these walls is an infinite dihedral group. In particular, $R$ is infinite. 
    As  $\lambda$ is unbounded, there exist a pair of non-intersecting walls which each intersect $\lambda$ transversely. In particular, $L$ generates an infinite group as well.
    By Proposition~\ref{prop:contained_in_wide}, $\langle R \cup L \rangle$ is contained in a wide subgroup.
    The proposition now follows.
\end{proof}

\subsection{Finding small bridges}
This subsection is dedicated to the proof the following proposition:
\begin{proposition} \label{prop:bridges_not_large}
Let $W_\Gamma$ be an affine-free Coxeter group, $\Sigma$ its Davis complex and $K$ a non-negative integer.
Then there exists a number $M = M(\Gamma, K)$, depending only on $\Gamma$ and $K$ such that the following holds:

 Given a geodesic ray $\gamma$ in $\Sigma$ with the property that any subpath of $\gamma$ of length at least $K$ does not have its label in a wide subgraph 
 and any pencil $H_1, \dots, H_{N}$ intersecting $\gamma$ with $N \ge M$, then the bridge between $H_1$ and $H_{N}$ is not large.
\end{proposition}

Proposition~\ref{prop:bridges_not_large} follows from the next technical proposition.
\begin{proposition} \label{prop:van_kampen}
    Let $W_\Gamma$ be an affine-free Coxeter group and $\Sigma$ be its Davis complex.
    There exist constants $\epsilon, \kappa > 0$, depending only on $\Gamma$, such that the following holds.
    Let $D$ be a Van-Kampen diagram over $\Sigma$ with boundary path $\lambda \alpha \rho^{-1} \beta^{-1}$ such that:
    \begin{enumerate}
	    \item $\lambda$, $\alpha$, $\rho$ and $\beta$ are geodesic. 
	    \item $\text{Label}(\beta)$ is contained in a wide subgraph.
	    \item $R > \kappa$ dual curves intersect both $\beta$ and $\alpha$. 
    \end{enumerate}
    Then, $\alpha$ has a subpath of length $\lfloor \epsilon R \rfloor$ whose label is contained in a wide subgraph.
\end{proposition}

We first prove Proposition~\ref{prop:bridges_not_large} from Proposition~\ref{prop:van_kampen}. Then, we turn to the proof of Proposition~\ref{prop:van_kampen}.

\begin{proof}[Proof of Proposition~\ref{prop:bridges_not_large}]
Set $M := \lfloor \frac{K}{\epsilon} +1 \rfloor$, where $K$ is as in the statement of the proposition and $\epsilon$ is as in Proposition~\ref{prop:van_kampen}.
Let $\gamma$ be a geodesic ray with the property that any subpath of length at least $K$ does not have its label in a wide subgraph.
Let $H_1, \dots, H_N$ be a pencil of walls intersecting $\gamma$ with $N \ge M$.
Suppose for a contradiction that the bridge $B$ between $H_1$ and $H_N$ is large. 
Let $\beta'$ be a minimal length CAT(0) geodesic between $H_1$ and $H_N$.
Let $u$ (resp. $v$) be a closest vertex to the startpoint (resp. endpoint) of $\beta'$.
Moreover, we can assume that $u$ (resp. $v$) is not in the same component of $\Sigma \setminus H_1$ (resp. $\Sigma \setminus H_N$) as $H_N$ (resp. $H_1$).
Let $\beta$ be a (Cayley graph) geodesic from $u$ to $v$. 
By Proposition~\ref{prop:contained_in_wide}, the set of reflections associated to the walls that intersect $\beta'$ transversely, is contained in a wide subgroup. As the walls that intersect $\beta$ are the same as the walls that intersect $\beta'$ transversely, $\text{Label}(\beta)$ is contained in a wide subgroup by Lemma~\ref{lem:in_wide}.

Let $\alpha$ be the subpath of $\gamma$ from the edge containing $H_1 \cap \gamma$ to the edge containing $H_N \cap \gamma$ which contains these edges.
Let $\lambda$ (resp. $\rho$) be a geodesic from the startpoint (resp. endpoint) of $\beta$ to the startpoint (resp. endpoint) of $\alpha$. 
Note that $M$ dual curves (for instance: $H_1, \dots, H_M$) intersect both $\alpha$ and $\beta$.
By Proposition~\ref{prop:van_kampen}, there is a subpath of $\alpha$ of length $\lfloor \epsilon M \rfloor = \lfloor \epsilon \lfloor \frac{K}{\epsilon} +1 \rfloor \rfloor \ge K$ whose label is contained in a wide subgraph. 
However, this contradicts our hypothesis.
\end{proof}

The remainder of this subsection is devoted to proving Proposition~\ref{prop:van_kampen}, and 
\emph{we fix the notation of Proposition~\ref{prop:van_kampen} for the remainder of this subsection.}

The following lemma follows from Theorem 5.2 and Proposition 5.3 in \cite{Mihalik-Tschantz} and gives us a decomposition of dual curves intersecting $\alpha$:

\begin{lemma}[Dilworth's theorem]
     Let $\mathcal L$ (resp. $\mathcal B$, $\mathcal R$) be the set of dual curves in $D$ which intersect both $\alpha$ and $\lambda$ (resp. $\beta$, $\rho$). Then there is a number $F_{\Gamma}$, depending only on $\Gamma$, such that $\mathcal L = \mathcal L_1 \sqcup \dots \sqcup \mathcal L_{l}$,  $\mathcal B = \mathcal B_1 \sqcup \dots \sqcup \mathcal B_b$ and $\mathcal R = \mathcal R_1 \sqcup \dots \mathcal R_{r}$ with $l, b, r \le F_\Gamma$ and such that, for all $1 \le i \le k$ (resp. $1 \le i \le b$, $1 \le i \le r$), any two walls in $\Sigma$ corresponding to dual curves in $\mathcal L_i$ (resp. $\mathcal B_i$, $\mathcal R_i$) do not intersect. 
\end{lemma}

We call the constant $F_\Gamma$ \emph{Dilworth's constant} and call each $\mathcal L_i$ (resp. $\mathcal B_i$, $\mathcal R_i$) a \emph{part} of $\mathcal L$ (resp. $\mathcal B$, $\mathcal R$).
As $\alpha$ is geodesic, no dual curve intersects it twice. 
Thus, every dual curve that intersects $\alpha$ is in exactly one of the sets $\mathcal L$, $\mathcal R$ or $\mathcal B$.

For $1 \le i \le l$, let $L^i_1, \dots, L^i_{l_i}$ denote the dual curves in the part $\mathcal L_i$.
We further assume that these dual curves are ordered with respect to the orientation of $\alpha$, i.e. for all $1 \le j < j' \le l_i$, $L^i_j \cap \alpha$ occurs before $L^i_{j'} \cap \alpha$, with respect to the orientation of $\alpha$.
Let $\Lambda_{\Gamma}$ be the ladder constant. 
We say that $L^i_j$ is a \emph{middle} dual curve if $\Lambda_{\Gamma} < j < l_i - \Lambda_{\Gamma}$,
and we say that it is a \emph{peripheral} dual curve otherwise.
We similarly define middle and peripheral dual curves for each part in $\mathcal R_i$ and in $\mathcal B_i$. 

Let $H$ be a dual curve in $D$ that intersects $\alpha$.
We say that $H$ is a \emph{middle} (resp. \emph{peripheral}) dual curve, if it is a middle (resp. peripheral) dual curve for some part in $\mathcal L$, $\mathcal R$ or $\mathcal B$.
This is well defined as each such dual curve is in a distinct set in $\{\mathcal L, \mathcal R, \mathcal B\}$ and in a distinct part of that set.
The next lemma is immediate:
\begin{lemma} \label{lem:peripheral_bound}
    The number of dual curves that intersect $\alpha$ and are peripheral is at most $6 F_\Gamma \Lambda_{\Gamma}$, where $\Lambda_{\Gamma}$ is the ladder constant and $F_\Gamma$ is the Dilsworth constant. 
\end{lemma}

Two dual curves $H$ and $K$ in $D$ are said to be in \emph{distinct sides} if they are both dual to $\alpha$, are 
not both in $\mathcal L$, are not both in $\mathcal R$ and are not both in $\mathcal B$.

The following lemma uses the affine-free hypothesis.
\begin{lemma} \label{lem:commuting_dual_curves}
	Let $H$ and $K$ be middle dual curves in $D$ that are in distinct sides. If $H \cap K \neq \emptyset$, then $H \cap K$ is contained in a square of $D$. In particular, the reflections corresponding to $H$ and $K$ commute.
\end{lemma}
\begin{proof}
	We prove the claim for when $H \in \mathcal L$ and $K \in \mathcal R$. The proof for the other cases is analogous.
	Let $i$ and $j$ be such that $H \in \mathcal L_i = \{L^i_1, \dots, L^i_{l_i}\}$ and $K \in \mathcal R_j = \{R^j_1, \dots, R^j_{r_j}\}$ where these sets are indexed with respect to the orientation of $\alpha$. 
	As $H$ and $K$ are middle dual curves, $H = L^i_{i'}$ for some $\Lambda_{\Gamma} < i' < l_i - \Lambda_{\Gamma}$ and $K = R^j_{j'}$ for some $\Lambda_{\Gamma} < j' < r_j - \Lambda_{\Gamma}$ where $\Lambda_{\Gamma}$ is the ladder constant.
	
	As $H \cap K \neq \emptyset$ and as dual curves in $\mathcal L_i$ (resp. $\mathcal R_j$) are pairwise non-intersecting, every dual curve in $A := \{L^i_{l_i - \Lambda_{\Gamma}}, \dots, L^i_{l_i}\}$ intersects every dual curve in $B := \{R^j_{1}, \dots, R^j_{\Lambda_{\Gamma}}\}$.
	Furthermore, $H \cap K$ is contained in the region of $D$ bounded by $\mathcal L^i_{l_i}$, $\mathcal R^j_1$ and $\alpha$.
	In particular, $H$ intersects every dual curve in $B$ and $K$ intersects every dual curve in $A$.

 \begin{figure}
  \begin{overpic}[scale=.8]{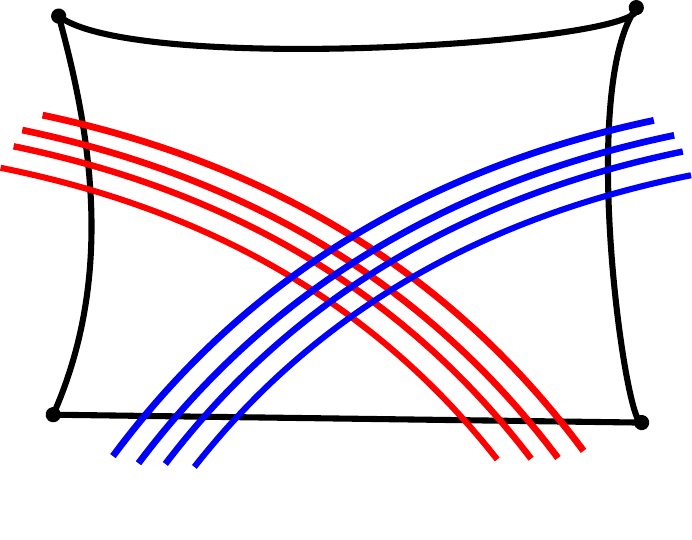}
  \put(34, 12){$K$}
  \put(14, 7){$R_1^j \enspace \ldots \enspace  R^j_{\Lambda_\Gamma}$}
  \put(65, 12){$H$}
  \put(71, 7){$L^i_{li-\Lambda_\Gamma} \> \ldots \> L^i_{l_i}$}
  \put(7, 45){$\lambda$}
  \put(50, 73){$\beta$}
  \put(95, 45){$\rho$}
  \put(50, 14){$\alpha$}
\end{overpic}  
\caption{Lemma \ref{lem:commuting_dual_curves}}
\end{figure}
	
	Suppose now for a contradiction that $H \cap K$ is contained in a polygon $c$ that is not a square. Let $Z$ be a dual curve dual to $c$ that is not equal to $H$ nor equal to $K$. As $Z$ cannot intersect $\alpha$ twice (as $\alpha$ is geodesic) it follows that either $Z$ intersects every dual curve in $A$ or it intersects every dual curve in $B$. 
	In the first case, a ladder is formed by the walls corresponding to $H$, $K$, $Z$ and the dual curves in $A$, and in the second case a ladder is formed by the walls corresponding to $H$, $K$, $Z$ and the dual curves in $B$.
	This is a contradicts Theorem~\ref{thm:recognizing_triangle_group} as the Davis complex of an affine-free Coxeter group cannot contain a ladder.
\end{proof}

\begin{lemma} \label{lem:intersection_implies_commuting1}
	Let $H$ and $K$ be middle dual curves in $D$ that intersect and are in distinct sides. Additionally, suppose that the edge $e$ containing $H \cap \alpha$ is adjacent to the edge $f$ containing $K \cap \alpha$. Then the labels of $e$ and $f$ commute.
\end{lemma}
\begin{proof}
	Let $e', f' \in \Sigma$ be the images of $e$ and $f$ respectively. 
	If $e' = f'$, then the lemma clearly follows. 
	On the other hand, if $e' \neq f'$, then they are distinct, adjacent edges of $\Sigma$.
	By Lemma~\ref{lem:commuting_dual_curves}, the reflections associated to $H$ and $K$ commute.
	Now, the labels of adjacent edges in $\Sigma$ commute if and only if there is an edge labeled by $2$ between the corresponding vertices in the  Coxeter diagram if and only if the reflections corresponding to the two edges commute.
	The lemma then follows.
\end{proof}

\begin{lemma} \label{lem:intersection_implies_commuting2}
	Let $\alpha'$ be a subpath of $\alpha$ that is not dual to any peripheral dual curve. Let $H$ and $K$ be dual curves intersecting $\alpha'$ that intersect each other and are in distinct sides. Then the label of the edge containing $H \cap \alpha'$ commutes with the label of the edge containing $K \cap \alpha'$.
\end{lemma}
\begin{proof}
	We prove the claim for when $H \in \mathcal L$. The proof for the other cases is analogous.
	Let $\alpha''$ be the smallest subpath of $\alpha$ containing both $H \cap \alpha$ and $K \cap \alpha$. 
	Let $\alpha'' = \alpha_1 \dots \alpha_n$ be a decomposition of $\alpha''$ such that for even $i$, $\alpha_i$ is only dual to dual curves in $\mathcal L$, and for odd $i$, $\alpha_i$ is only dual to dual curves in $\mathcal B \cup \mathcal R$. 
	Clearly, such a decomposition is possible and is unique.
	Furthermore, $n > 1$ as $H$ and $K$ are in distinct sides.

    Let $\alpha_1 = e_1 \dots e_x$ and $\alpha_2 = f_1 \dots f_y$ be the decomposition of the paths $\alpha_1$ and $\alpha_2$ into edges, and let $s_1, \dots, s_x$ and $t_1, \dots, t_y$ be the labels of the edges of $\alpha_1$ and $\alpha_2$ respectively.
 
	We prove the claim by induction on $n$. 
    Suppose first that $n = 2$. 
	By the structure of dual curves in $D$, any dual curve dual to $\alpha_2$ intersects a dual curve dual to $\alpha_1$.
	By Lemma~\ref{lem:intersection_implies_commuting1}, $s_x$ commutes with $t_1$. We can then form a new disk diagram by attaching a square to $s_x$ and $t_1$. 
	In the resulting diagram, the boundary path with label $s_1 \dots s_x t_1 \dots t_y$ has been replaced with the path with label $s_1 \dots s_{x-1}t_1 s_x t_2 \dots t_y$.
	Proceeding in this way by attaching squares, we can replace this boundary path with one with label $t_1 s_1 \dots s_x t_2 \dots t_y$ and we conclude that $t_1$ commutes with $s_i$ for all $1 \le i \le x$.
	Then, by iterating this process, we eventually replace $\alpha''$ in the original diagram with a path labeled by $t_1 \dots t_ys_1 \dots s_x$, and we conclude that $s_i$ commutes with $t_j$ for all $1 \le i \le x$ and $1 \le j \le y$. The claim now follows in this base case, as $H$ and $K$ are dual to $e_1$ and $f_y$ respectively, and the labels of these edges commute.
    \begin{figure}
        \begin{overpic}[scale=.3]{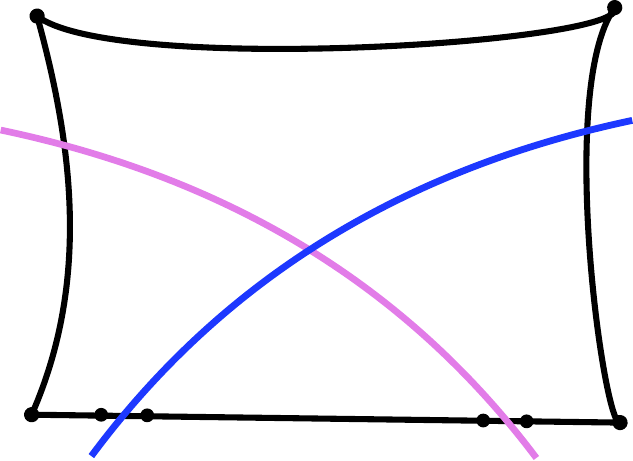}
            \put(19, -5){$e_1$}
            \put(78, -5){$f_y$}
            \put(-8, 47){$H$}
            \put(100, 48){$K$}
          \end{overpic}  
          \caption{Lemma \ref{lem:intersection_implies_commuting2}}
    \end{figure}
	
	Suppose now that $n > 2$ and the claim is true for $n-1$. By the same argument as the $n=2$ case above, we conclude that $s_i$ commutes with $t_j$ for all $1 \le i \le x$ and $1 \le j \le y$, and we can construct a new van-Kampen diagram where the initial subpath of $\alpha''$ with label $s_1 \dots s_x t_1 \dots t_y$ is replaced with a path with label $t_1 \dots t_y s_1 \dots s_x$.
	Let $\zeta$ be the subpath of this path with label $s_1 \dots s_x$. 
	Next, we apply the induction hypothesis to the path $\zeta \alpha_3 \dots \alpha_n$ (note that dual curves dual to $\zeta \alpha_3$ are all in $\mathcal R$) to conclude that $s_1$ commutes with the label of the edge containing $H \cap \alpha'$. The claim then follows.
\end{proof}

To simplify notation, we define the \emph{type} of a dual curve that intersects $\alpha$ to be the label of the edge of $\alpha$ dual to it.

\begin{lemma} \label{lem:subpath1}
    There exists a constant  $\delta >0$, depending only on $\Gamma$, and a subpath $\alpha' \subset \alpha$ such that
    \begin{enumerate}
    \item $\alpha'$ is not dual to any peripheral dual curves, 
    \item $\alpha'$ is dual to at least $\delta R$ dual curves in $\mathcal B$, 
     \item for any dual curve $H \in \mathcal L$ dual to $\alpha'$, there is a dual curve $H' \in \mathcal L$ dual to $\alpha$ such that $H$ and $H'$ have the same type and $H' \cap \alpha$ occurs after $\alpha'$ with respect to the orientation of $\alpha$, and
    \item for any dual curve $H \in \mathcal R$ dual to $\alpha'$, there is a dual curve $H' \in \mathcal R$ dual to $\alpha$ such that $H$ and $H'$ have the same type and $H' \cap \alpha$ occurs before $\alpha'$ with respect to the orientation of $\alpha$.
     \end{enumerate}
\end{lemma}
\begin{proof}
    Let $E$ be the set of edges in $\alpha$ that are dual to a peripheral curve. 
    By Lemma~\ref{lem:peripheral_bound}, $|E| \le 6F_\Gamma\Lambda_{\Gamma}$ where $\Lambda_{\Gamma}$ is the ladder constant and $F_\Gamma$ is the Dilsworth constant. Recall that $R = |\mathcal B|$.
	It follows that there is a subpath $\alpha'' \subset \alpha$  contained in $\alpha \setminus E$ and dual to at least $\lfloor \frac{R}{6F_\Gamma\Lambda_{\Gamma}} \rfloor$ dual curves in $\mathcal B$.
	Let $V = |V(\Gamma)| + 1$. 
	We partition $\alpha'' = \alpha_1 \dots \alpha_{V+1}$ so that for all $1 \le i < V+1$, $\alpha_i$ is dual to at least $\lfloor \frac{1}{V} \lfloor \frac{R}{6F_\Gamma\Lambda_{\Gamma}} \rfloor \rfloor$ dual curves in $\mathcal B$. 
	
	Note that each $\alpha_i$ satisfies (1). We now show that there exists some $1 \le k < V+1$ so that $\alpha_k$ additionally satisfies (3). 
	To see this, suppose for a contradiction that such a $k$ does not exist. 
	In particular, for all $1 \le i \le V$,  there is 
	a dual curve $H \in \mathcal L$ intersecting $\alpha_i$ of type $s_i$ such that any $H' \in \mathcal L$ intersecting $\alpha_{i+1} \dots \alpha_{V+1}$ is not of type $s_i$. 
	In particular, $s_i \neq s_{j}$ for all $i \neq j$.
	However, this is a contradiction as there are at most $V-1$ possible values for $s_i$. Thus, such a $k$ exists. 
	
    We can further partition $\alpha_k$ into $V+1$ subpaths, each dual to at least $\lfloor \frac{1}{V} \lfloor \frac{1}{V} \lfloor \frac{R}{6F_\Gamma\Lambda_{\Gamma}} \rfloor \rfloor \rfloor$ dual curves in $\mathcal B$. Proceeding similarly as in the previous paragraph, one of these subpaths $\alpha'$ must satisfy (4). Consequently, $\alpha'$ satisfies (1)--(4) by picking $\delta(V, \Gamma)$ so that $\delta R < \lfloor \frac{1}{V} \lfloor \frac{1}{V} \lfloor \frac{R}{6F_\Gamma\Lambda_{\Gamma}} \rfloor \rfloor \rfloor$ 
\end{proof}

\begin{lemma} \label{lem:subpath2}
    Let $\alpha' \subset \alpha$ be a subpath as given by Lemma~\ref{lem:subpath1}. Then, 
     any two dual curves in distinct sides that intersect $\alpha'$ have commuting types.
\end{lemma}
\begin{proof}
    Let $H$ and $K$ be dual curves in distinct sides that intersect $\alpha'$.
    Up to relabeling $H$ and $K$, we can assume that $H \in \mathcal L \cup \mathcal R$.
    We show the claim for when $H \in \mathcal L$. The proof is similar for when $H \in \mathcal R$. 
    
    As $H$ and $K$ are in distinct sides and as $H \in \mathcal L$, $K \in \mathcal R \cup \mathcal B$.  
    Let $s$ and $t$ be the types of $H$ and $K$ respectively.
	If $H \cap K \neq \emptyset$, then $s$ commutes with $t$ by Lemma~\ref{lem:intersection_implies_commuting2}.
	On the other hand, if $H \cap K = \emptyset$, then by (3) there exists a $H' \in \mathcal L$ of type $s$ dual $\alpha$ so that $H' \cap \alpha$ occurs after $\alpha'$.
	As $K \in \mathcal R \cup \mathcal B$ and as $H' \in \mathcal L$, $H'$ intersects $K$. Consequently, $s$ commutes with $t$ by Lemma~\ref{lem:intersection_implies_commuting2}.
\end{proof}

We are now ready to prove the main proposition of this subsection:

\begin{proof}[Proof of Proposition~\ref{prop:van_kampen}]
	Choose $\alpha' \subset \alpha$ as in Lemma~\ref{lem:subpath1}, and let $w$ be the label of $\alpha'$. 
	Let $S_L$ (resp. $S_R$, $S_B$) be the set of vertices in $\Gamma$ that are the type of some dual curve in $\mathcal L$ (resp. $\mathcal R$, $\mathcal B$) that intersects $\alpha'$. 
	By Lemma~\ref{lem:subpath2}, every vertex in $S_L \cup S_R$ commutes with every vertex in $S_B$. 
	Thus, there is an expression $w_1w_2$ for $w$ such that $\text{Label}(w_1) = S_L \cup S_R$ and $\text{Label}(w_2) = S_B$.
	As in Definition~\ref{def:constants}, let $M_\Gamma$ denote the smallest constant such that $W_{\text{Label}(h)}$ is infinite for any geodesic word $h$ with $|h| > M_\Gamma$. 
	By choosing $\kappa$ large enough, we can guarantee that $\delta R > M_\Gamma$ where $\delta$ is as in Lemma~\ref{lem:subpath1}. Note that this ``large enough'' only depends on $\Gamma$.  
	With this choice of $R$, we have that $|w_2| > M_\Gamma$ and $W_{\text{Label}(w_2)}$ is an infinite subgroup.

    There are now two cases. 
	If $|w_1| > M_\Gamma$ then $W_{S_L \cup S_R}$ is infinite and $W_{S_L \cup S_R} \times W_{S_B}$ is a wide subgroup.
	On the other hand, if $|w_1| < M_\Gamma$, then $\alpha'$ contains a subpath $\alpha''$ of length  
	$\frac{|\alpha'|}{M_\Gamma}$ such that any dual curve dual to it is in $\mathcal B$. 
	As $\text{Label}(\beta)$ is in a wide subgroup by hypothesis and as every dual curve intersecting $\alpha''$ intersects $\beta$, $\text{Label}(\alpha'')$ is in a wide subgroup by Lemma~\ref{lem:in_wide}.
    Thus the claim follows in either case by setting $\epsilon = \delta / M_\Gamma$.
\end{proof}

\subsection{Finding nice pencils}

In this section, we show how to find a ``nice'' pencil intersecting a geodesic ray that spends a bounded amount of time in wide subgraphs. The main result is the following proposition:
\begin{proposition} \label{prop:best_pencil}
Let $W_\Gamma$ be an affine-free Coxeter group and $\gamma$ a geodesic segment in the Davis complex $\Sigma = \Sigma_\Gamma$ with the property that any subpath of $\gamma$ of length at least $K$ does not have its label contained in a wide subgraph. Then, there is an integer $C = C(\Gamma, K)$ such that $\gamma$ intersects a pencil $H_1, K_1, H_2, K_2, \dots, H_{N}, K_N$ of walls satisfying:
\begin{enumerate}
    \item $N := \lfloor \frac{|\gamma|}{C} \rfloor$,
    \item The bridge between $H_i$ and $K_{i}$ is not large for each $1 \le i \le N$, and 
    \item $d(H_i, K_i) \le C$
\end{enumerate}
\end{proposition}

We first need to prove two lemmas which build up to the proposition. The first lemma follows from Dilworth's theorem.

\begin{lemma} \label{lem:large_pencils}
Let $W_\Gamma$ be a Coxeter group and $\Sigma$ be its Davis complex.
There exists a number $C_1 = C_1(\Gamma)$, such that any geodesic $\gamma$ in $\Sigma$ intersects a pencil of $\lfloor \frac{|\gamma|}{C_1} \rfloor$ walls.
\end{lemma}

The next lemma uses the main result from the previous subsection.

\begin{lemma} \label{lem:finding_pencils}
Let $W_\Gamma$ be an affine-free Coxeter group and $\gamma$ a geodesic segment in the Davis complex $\Sigma = \Sigma_\Gamma$ with the property that any subpath of $\gamma$ of length at least $K$ does not have its label contained in a wide subgraph. Then, there is an integer $C_2 = C_2(\Gamma, K)$ such that $\gamma$ intersects a pencil $H_1, \dots, H_{N}$ of $N := \lfloor \frac{|\gamma|}{C_2} \rfloor$ walls with the property that, for all $1 \le i < N$, the bridge between $H_i$ and $H_{i+1}$ is not large.  
\end{lemma}
\begin{proof}
 By Lemma~\ref{lem:large_pencils}, there there exists a number $C_1 = C_1(\Gamma)$ and a pencil $Q_1, \dots, Q_{N'}$ of $N' := \lfloor \frac{|\gamma|}{C_1} \rfloor$ walls intersecting $\gamma$. 
 By Proposition~\ref{prop:bridges_not_large}, there is a constant $M = M(\Gamma, K)$ such that for any pencil $P_1, \dots, P_M$ of $M$ walls intersecting $\gamma'$, the bridge between $P_1$ and $P_M$ is not large.
In particular, the bridge between $Q_{iM}$ and $Q_{(i+1)M}$ is not large for each integer $1 \le i < \lfloor \frac{N'}{M} \rfloor$.
Thus, the pencil $H_1 = K_M, H_2 = K_{2M}, \dots, H_N = K_{\lfloor \frac{N'}{M} \rfloor}$ has the property that the bridge between any two consecutive walls is not large. Note that $N = \lfloor \frac{N'}{M} \rfloor = \lfloor \frac{\lfloor \frac{|\gamma|}{C_1} \rfloor}{M} \rfloor$ where $M$ and $C_1$ only depend on $\Gamma$ and $K$. Thus, the constant $C_2$ exists. 
\end{proof}

We are now ready to prove the main result of this subsection:

\begin{proof}[Proof of Proposition~\ref{prop:best_pencil}]
By Lemma~\ref{lem:finding_pencils} there is a constant $C_2 = C_2(\Gamma, K)$ and a pencil $Q_1, \dots, Q_{N'}$ of walls intersecting $\gamma$ such that $N' = \lfloor \frac{|\gamma|}{C_2} \rfloor$ and the bridge between $Q_i$ and $Q_{i+1}$ is not large for each $1 \le i < N'$.
There are less than $\frac{N'}{4}$ values of $i \in \{1, \dots, \lfloor \frac{N'}{2} \rfloor \}$ such that $d(Q_{2i}, Q_{2i+1}) > 5C_2$.
For otherwise we get the contradiction $|\gamma| \ge \frac{N'}{4} 5C_2 = \lfloor \frac{|\gamma|}{C_2} \rfloor \frac{5 C_2}{4} > |\gamma|$.
Thus, there is a subsequence, $H_1 = Q_{i_1}, K_1 = Q_{i_1 + 1}, H_2 = Q_{i_2}, K_2 = Q_{i_2 + 1}, \dots, H_N = Q_{i_N}, K_N = Q_{i_N + 1}$ so that $d(H_i, H_{i+1}) \le 5C_2$ and $N = \lfloor \frac{N'}{4} \rfloor = \lfloor \frac{\lfloor \frac{|\gamma|}{C_2} \rfloor}{4} \rfloor$. The claim then follows by choosing $C$ large enough.
\end{proof}

\subsection{Divergence of walls and bridge diameter bound}
Given a pair of non-intersecting walls that are bounded distance apart, we bound the diameter of the bridge (if finite) between them in Lemma~\ref{lem:bridge_diam_bound} below, and we give a lower bound on the rate they can diverge in Lemma~\ref{lem:wall_divergence_bound}.

\begin{lemma} \label{lem:bridge_diam_bound}
Let $W_\Gamma$ be a Coxeter group, let $\Sigma$ be its Davis complex. Given any constant $C >0$, there exists a constant $D = D(C, \Gamma)$ such that the following holds. Let $H$ and $K$ be distinct, non-intersecting walls in $\Sigma$ such that $d(H, K) < C$, and let $B$ be the bridge between $H$ and $K$.
Then either $B$ is unbounded or it has diameter less than $D$. 
\end{lemma}
\begin{proof}
As the $1$-skeleton of $\Sigma$ is locally finite, given a wall $H$, there are only a finite number of walls $K$ such that $d(H,K) < C$. 
Consequently, as there are finitely many orbits of walls in $\Sigma$, there are finitely many orbit classes of bridges between pairs of walls that are distance less than $C$ apart. 
Let $\mathcal B$ be the set of all such bridges. We can then take $D = \max\{B \in \mathcal B ~|~ \text{diam}(B) < \infty \}$.
\end{proof}

The next lemma follows as walls are convex and as the distance function $d(~\cdot~,H)$ is convex where $H$ is a convex subcomplex in a CAT(0) space. 
The proof of  \cite[Lemma 2-10(4) and Remark 2.13]{Huang} can readily be adapted to prove this lemma.

\begin{lemma} \label{lem:wall_divergence_bound}
For any integer $C$, there is a constant $\epsilon$ depending only on $\Gamma$ and $C$, such that if given any two walls $H$ and $K$ in $\Sigma$ with $d(H, K) \le C$, and any two points $x \in H$ and $y \in K$, then
	\[d(x, y) \ge  \epsilon \min\{d(x, B), d(y,B)\} \]
	where $B$ is the bridge between $H$ and $K$.
\end{lemma}

\subsection{Proof of Theorem~\ref{thm:morse_char}}
In this subsection, we piece together the results from the previous subsections in order to prove Theorem~\ref{thm:morse_char}.
We first recall the following definition from \cite{Charney-Sultan}.

\begin{definition}
Let $\gamma$ be a quasi-geodesic in a metric space. Given $t > r > 0$, we let $\rho_\gamma(r,t)$ denote the infimum of the lengths paths from $\gamma(t-r)$ to $\gamma(t + r)$ which do not intersect the open ball of radius $r$ based at $\gamma(t)$. The \emph{lower divergence} of $\gamma$ is the function $ldiv_\gamma(r) = \inf_{t>r} \rho_\gamma(r, t)$.
\end{definition}

We show that geodesic rays that spend bounded time in wide subgroups must have quadratic lower divergence.
\begin{proposition} \label{prop:lower_div}
Let $W_\Gamma$ be an affine-free Coxeter group and $\gamma$ a geodesic ray in the Davis complex $\Sigma = \Sigma_\Gamma$ with the property that any subpath 
of $\gamma$ of length at least $K$ does not have its label contained in a wide subgraph. 
Then, there is a constant $A >0$ such that for all $r$ large enough, $ldiv_\gamma(r) > Ar^2$.
\end{proposition}
\begin{proof}
Let  $\gamma(t): [0, \infty) \to \Sigma$ be a geodesic ray in $\Sigma$ with the property that any subpath of $\gamma$ larger than $K$ is not in a wide subgroup.
Fix integers $r > t > 0$. Without loss of generality, we assume that $r$ is even.
Consider the subpath $\gamma'$ of $\gamma$ based at $\gamma(t)$ and of length $\frac{r}{2}$.

We now define some relevant walls, bridges and constants that depend only on $K$ and $\Gamma$.
By Proposition~\ref{prop:best_pencil}, there is a constant $C = C(\Gamma, K)$ and a pencil $H_1, K_1, \dots, H_N, K_N$ of walls intersecting $\gamma'$ such that $N = \lfloor \frac{|\gamma'|}{C} \rfloor = \lfloor \frac{r}{2C} \rfloor $, the bridge between $H_i$ and $K_i$ is not large and $d(H_i, K_i) \le C$.
By Lemma \ref{lem:bridge_diam_bound}, there is a constant $D = D(\Gamma, C)$ such that that the bridge $B_i$ between $H_i$ and $K_i$ has diameter at most $D$.  
Finally, by Lemma \ref{lem:wall_divergence_bound}, there is a constant $\epsilon = \epsilon(\Gamma, C) > 0$ 
so that for any $1 \le i \le N$, $x \in H_i$ and $x' \in K_i$, we have that $d(x, x') \ge \epsilon \min \{d(x, B_i), d(x', B_i) \}$. 
Without loss of generality, we can assume that $0 < \epsilon < 1$.

Let $\alpha$ be a minimal length path from $\gamma(t - r)$ to $\gamma(t + r)$ that does not intersect the open ball of radius $r$ about $\gamma(t)$. 
We may assume such an $\alpha$ exists, otherwise $\gamma$ has infinite lower divergence and we are done.
For each $1 \le i \le N$, define $x_i := \gamma \cap H_i$, $x_i' = \gamma \cap K_i$ and choose $y_i \in \alpha \cap H_i$ and $y_i' \in \alpha \cap K_i$.

We bound the distance between $x_i$ (resp. $x_i'$) and $B_i$.
As $d(x_i, x_{i}') \le C$, we have that $\min \{d(x_i, B_i), d(x_i', B_i) \} \le \frac{C}{\epsilon}$. In particular, both $d(x_i, B_i)$ and $d(x_i', B_i)$ are no larger than $\frac{C}{\epsilon} + C \le \frac{2C}{\epsilon}$.

We now given an upper bound the distance between $y_i$ and $B_i$:
\begin{align*}
d(y_i, B_i) \ge d(x_i, y_i) - d(x_i, B_i) - D &\ge d(x_i, y_i) - \frac{2C}{\epsilon} -D \\
&\ge d(\gamma(t), y_i) - d(x_i, \gamma(t)) -  \frac{2C}{\epsilon} - D  \\
&\ge r - \frac{r}{2} -  \frac{2C}{\epsilon} - D \\
&\ge \frac{r}{2} - \frac{2C}{\epsilon}  - D
\end{align*}
Similarly, we obtain that $d(y_i', B_i) \ge \frac{r}{2} - \frac{2C}{\epsilon}  - D$.

By Lemma~\ref{lem:wall_divergence_bound} and the above inequality, we have 
\[d(y_i, y_i') \ge \epsilon \min \{ d(y_i, B_i), d(y_i', B_i) \} \ge \frac{\epsilon r}{2} - 2C - \epsilon D\]

Let $\alpha_i$ be the subpath of $\alpha$ from $y_i$ to $y_{i}'$.
We are now ready to give a lower bound on the length of $\alpha$:
\[|\alpha| \ge\sum_{i=1}^{N} |\alpha_i| \ge\sum_{i=1}^{N} d(y_i, y_{i+1}) \ge \frac{N\epsilon r}{2} - (2C - \epsilon D)N = 
\Bigl \lfloor \frac{r}{2C} \Bigr \rfloor \epsilon r \frac{1}{2} - (2C - \epsilon D) \Bigl \lfloor \frac{r}{2C} \Bigr \rfloor \]
Thus, we can choose some $A > 0$, depending only on $C$, $D$ and $\epsilon$ (which only depend on $K$ and $\Gamma$), so that $|\alpha| \ge Ar^2$ for all $r$ large enough. The claim now follows.
\end{proof}

We are now ready to prove the main theorem of this section.

\begin{proof}[Proof of Theorem~\ref{thm:morse_char}]
Let $\gamma$ be a geodesic ray. 
To see one direction, suppose that there exists a constant $K$ such that any subpath of $\gamma$ larger than $K$ does not have its label contained in a wide subgraph.
Then, $\gamma$ has quadratic lower divergence by Proposition~\ref{prop:lower_div}. In particular, it is Morse by \cite{Charney-Sultan}. 

The other direction is easier. 
If $\gamma$ has arbitrarily long subpaths with labels contained in wide subgraph, then arbitrarily long subpaths of $\gamma$ are contained in convex, product subcomplexes with unbounded factors. In particular, $\gamma$ is not~Morse.    
\end{proof}

\section{Filter Construction} \label{sec:filters}

\subsection{Fans}
Let $\gamma$ be a path in a van-Kampen diagram, and let $\overline{\gamma} = s_0 \dots s_n$ be the
word corresponding to $\gamma$.
Let $j$ be the smallest integer such that $\text{Label}(s_j \dots s_n) \subset \Delta$ for some wide subgraph $\Delta \subset \Gamma$. 
If such a $j$ exists, we call the terminal subpath of $\gamma$ labeled by $s_j \dots s_n$ the \emph{wide tail of $\gamma$}. 
If no such $j$ exists, we define the wide tail of $\gamma$ to be the terminal vertex of $\gamma$ (a length $0$ path).

Recall that $M_\Gamma$ is smallest constant such that $W_{\text{Label}(w)}$ is infinite for any
geodesic word $w$ with $|w| > M_\Gamma$.

We now define a \emph{fan} below. Refer to Figure~\ref{fig_fan} for an example.

\begin{definition} \label{def:cox_fans}
    A \emph{fan} $F$ is a planar van-Kampen diagram consisting of a path $\gamma$, a set of at least $3$ edges $f_0, \dots, f_r$ adjacent to the endpoint of $\gamma$ and, for each $0 \le i < r$ a cell $c_i$ containing both $f_i$ and $f_{i+1}$. Moreover, the following must be satisfied:
	\begin{enumerate}
	    \item $\overline{\gamma f_i}$ is geodesic for each $0 \le i \le r$. 
		\item Let $\gamma'$ be the wide tail of $\gamma$. Either 
		\begin{enumerate}
		    \item $|\gamma'| \le M_\Gamma$, or
		    \item There is a wide subgraph $\Delta$ such that $\text{Label}(\gamma') \subset \Delta$
		    and the label of $f_i$ is not in $\Delta$ for each $0 < i < r$. 
		\end{enumerate}
	\end{enumerate}
	We call $\gamma$ the \emph{base path} of $F$ and $f_0, \dots, f_r$ the \emph{fan edges} of $F$.
	The edge $f_0$ (resp. $f_r$) is called the \emph{left fan edge} (resp. \emph{right fan edge}) of $F$. 
	The edges $f_1, \dots, f_{r-1}$ are called \emph{interior fan edges}. 
	
	Each $c_i$ is a $2m_{st}$-gon where where $s$ and $t$ are respectively the labels of $f_i$ and $f_{i+1}$ (and $m_{st}$ is the label of the edge $(s,t) \subset \Gamma$). 
	Let $v$ be the endpoint of $\gamma$.
	There is a vertex $v_i \in C_i$ that is furthest from $v$ in $c_i$, which we call the \emph{top vertex of $c_i$}.
	Let $\lambda_i$ and $\rho_i$ be the two simple paths in $c_i$ from $v$ to $v_i$. 
	As $F$ is planar, we think of $\lambda_i$ (resp. $\rho_i$) as lying to the left (resp. right) of $\gamma$ with respect to its orientation.
	We call the last edge $t_i$ of $\lambda_i$ the \emph{top left edge of $c_i$}.
	The \emph{top left edges of $F$} is the set $\{t_0, \dots, t_{r-1} \}$.
	
	The edges of a fan can be given a natural orientation. 
	Namely, we give edges in $\gamma$ the orientation inherited from $\gamma$, and for each $1 \le i < r$, we orient the edges of $c_i$ with respect to the orientations given by $\lambda_i$ and $\rho_i$.
	It readily follows that this gives a consistent orientation of all edges of $F$.
\end{definition}

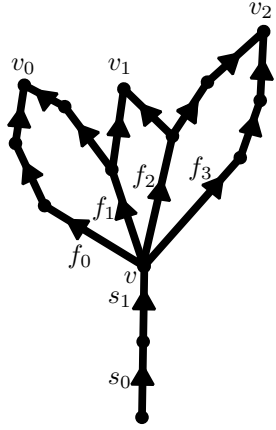
\begin{figure} \label{fig_fan}
  \begin{overpic}[percent]{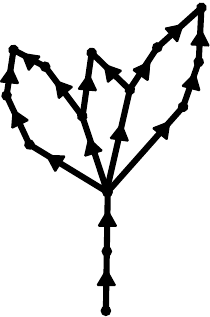}
  \put(29,35){$v$}
  \put(25, 10){$s_0$}
  \put(25, 30){$s_1$}
  \put(15, 40){$f_0$}
  \put(21, 52){$f_1$}
  \put(31, 60){$f_2$}
  \put(45, 62){$f_3$}
  \put(1,88){$v_0$}
  \put(25,88){$v_1$}
  \put(60,102){$v_2$} 
\end{overpic}
\caption{Illustration of a fan}    
\end{figure}

We will show how to construct fans in wide-spherical-avoidant graphs. However, we first need the following lemma regarding Coxeter groups:
\begin{lemma} \label{lem:forbidden_last_letter}
    Let $w = s_1 \dots s_m$ be a geodesic word in the Coxeter group $W_\Gamma$.
    Let $s, t \in V(\Gamma)$ be non-adjacent vertices such that $s_i = s$ for some $i$ and $s_j \neq t$ for all $j \ge i$.
    Then, no geodesic expression for $w$ ends with~$t$. 
\end{lemma}
\begin{proof}
    Suppose, for a contradiction, that $w'$ is a geodesic expression for $w$ that ends with $t$. By Tit's solution to the word problem \cite{Tits}, 
    there is a sequence of words
    \[w = w_1, w_2, \dots, w_k = w'\]
   where $w_{i+1}$ is obtained by the ``Tits move'' that replaces a subword $aba \dots$ (alternating in $a$ and $b$) of $w_i$ of length $m_{ab}$ with its expression $bab \dots$ of the same length, where $m_{ab}$ is the label of the edge $(a,b)$ in $\Gamma$.
   
    Note that each $w_i$ must have an occurrence of the letter $s$.
    Thus, for some $1 \le r < k$, the word $w_r$ has no occurrence of the letter $t$ after the last occurrence of the letter $s$, and the word $w_{r+1}$ has an occurrence of the letter $t$ after the last occurrence of the letter $s$.
    But this means that $w_{r+1}$ is obtained from $w_r$ by replacing a word either of the form $sts \dots$ or of the form $tst \dots$.  
    This is a contradiction as $s$ and $t$ are not adjacent in $\Gamma$ and so no such Tits move can be made.
\end{proof}

The following lemma summarizes some well known facts regarding Coxeter groups. For proofs, we refer the reader to Bjorner and Brenti's book \cite[Corollary~1.4.6]{Bjorner-Brenti} and Davis' book \cite[Lemma~4.7.2]{Davis}.

\begin{lemma}\label{lem:ending_letters}
    Let $W_\Gamma$ be a Coxeter group, and let $w$ be a geodesic word in $W_\Gamma$. Let $K \subset V(\Gamma)$ be the set of all letters $s$ such that some geodesic expression for $w$ ends with $s$. Then $W_K$ is a finite group. Moreover, $wt$ is geodesic for each $t \in V(\Gamma) \setminus K$.
\end{lemma}

\begin{lemma} \label{lem:cox_fan_existence}
	Let $W_\Gamma$ be a one-ended, wide-spherical-avoidant Coxeter group.
	Let $F'$ be a Van-Kampen diagram over $W_\Gamma$ consisting of a path $\gamma$ and two edges, $x$ and $y$, each adjacent to the endpoint of $\gamma$.
	Moreover, suppose that $\overline{\gamma x}$ and $\overline{\gamma y}$ are geodesic.
	Then there is a fan $F$, containing $F'$ as a subdiagram,  with fan edges $x = f_0, f_1, \dots, f_r = y$ and base path $\gamma$.
\end{lemma}
\begin{proof}
    Let $e_0 \dots e_n$ be the edges of $\gamma$, and let $t_i$ be the label of $e_i$.
    Let $s$ and $t$ be the labels of $x$ and $y$ respectively. 
	Let $V' \subset V(\Gamma)$ be the set of letters for which a geodesic expression for the label of $\gamma$ can end in, and let $K$ be the subgraph induced by $V'$. 
	$W_{K}$ is finite by Lemma~\ref{lem:ending_letters}.
	
	Suppose first that the wide tail of $\gamma$ has length less than $M_\Gamma$, where $M_\Gamma$ is as in Definition~\ref{def:constants}.
	As $\overline{\gamma x}$ and $\overline{\gamma y}$ are geodesic, $s$ and $t$ are not in  $K$.
	As $W_\Gamma$ is one-ended, $K$ does not separate $\Gamma$, and consequently there is a path $\alpha \subset \Gamma \setminus K$ from $s$ to $t$.
        We would like the length of $\alpha$ to be greater than $1$. If the length of the chosen $\alpha$ is $1$, then $x$ and $y$ are adjacent. In this case, we replace $\alpha$ with a new path traversing the vertices $x, y, x, y$. Let $x = s_0, \dots, s_r = y$ be the vertices traversed by $\alpha$.
	We construct a fan $F$ with base path $\gamma$, fan edges $x = f_0, \dots, f_r = y$ and cells $c_0, \dots, c_{r-1}$ such that each $f_i$ has label $s_i$ and
	$c_i$ is a $(2m_{s_i s_{i+1}})$-gon containing both $f_i$ and $f_{i+1}$.
	Note that, by Lemma~\ref{lem:ending_letters} and as $s_i \notin K$, $\overline{\gamma f_i}$ is geodesic for each $i$.
	The claim is thus proven for this case.
	
	Suppose, on the other hand, that $\gamma$ has a wide tail $\gamma' = e_l \dots e_n$ with $|\gamma'| \ge M_\Gamma$. Let $\Delta = \Delta_1 \star \Delta_2$ be a wide subgraph associated to the wide tail $\gamma'$.
	For $j \in \{1,2\}$, let $C_j = \{t_l, \dots, t_n\} \cap V(\Delta_j)$.
	As $\{t_l, \dots, t_n\}$ is not contained in a finite subgroup (by definition of $M_\Gamma$) and as every letter in $C_1$ commutes with every letter in $C_2$, either $W_{C_1}$ or $W_{C_2}$ is infinite.
	Up to relabeling, we assume that $W_{C_1}$ is infinite.
	
	Let $K'$ be the subgraph induced by $V(K) \setminus (V(K) \cap \{t_l, \dots, t_n\}$).
	By Lemma~\ref{lem:forbidden_last_letter}, every vertex in $K'$ is adjacent in $\Gamma$ to $t_i$ for all $i \ge l$.
	Thus, $(C_1, \Delta_2, K')$ is a special join.
	As $\Gamma$ is wide-spherical-avoidant, there is a path $\alpha \subset \Gamma$ from $s$ to $t$ such that $\alpha \cap (K' \cup C_1 \cup \Delta_2) \subset \{s, t\}$. 
	Let $s = s_0, \dots, s_r = t$ be the vertices traversed by $\alpha$ ordered with respect to $\alpha$'s orientation. 
	Without loss of generality, we can assume that 
	$s_i \notin \{s, t\}$ for all $1 < i < m$.
	Also note that, like in the previous case, we can choose $\alpha$ so that $r \ge 2$.
	As in the previous case, we construct a fan with fan edges $e = f_0, f_1, \dots, f_r =f$ where $f_i$ has label~$s_i$.

	It just remains to check that $t_0 \dots t_m s_i$ is geodesic for each $i$. 
	Note that $K \subset V(K') \cup V(C_1) \cup V(\Delta_2)$.
	In particular, $s_i \notin K$ for each $1 < i < m$.
	Thus, $t_0 \dots t_m s_i$ is indeed geodesic for each $i$ by Lemma~\ref{lem:ending_letters}. The claim now follows.
\end{proof}

The proof of the following lemma is outlined within the proof of \cite[Lemma~3.9]{Mihalik-Tschantz}. We give a proof here for completeness.

\begin{lemma} \label{lem:cox_fan_geodesics}
	Let $F$ be a fan with base path $\gamma$. Let $\gamma'$ be a directed path in $F$ starting at the endpoint of $\gamma$. Then $\gamma \gamma'$ is geodesic.
\end{lemma}
\begin{proof}
    By the structure of a fan, $\gamma'$ is contained in a polygon $c$. Let $s$ and $t$ be the two vertices in $\Gamma$ that appear as labels of edges in $c$. Let $P$ be the subgraph of $\Gamma$ consisting of $s$, $t$ and the edge between $s$ and $t$.
    
    Let $\Sigma$ be the Davis complex of $W_\Gamma$. Let $w$ be the label of $\gamma$, and let $\hat{\gamma}$ be the path in $\Sigma$ based at the identity vertex $b$ and with label the same as that of $\gamma$. 
    Let $q$ be the endpoint of $\hat{\gamma}$.
    Let $\hat{\gamma}'$ be the path in $\Sigma$ with starting at $q$ and with label the same as $\gamma'$.
    Note that $\hat{\gamma}'$ is contained in $wW_P$ (when thought of as a subset of $\Sigma$).

	Let $\alpha$ be a geodesic in $\Sigma$ from $b$ to $wW_P$. 
	Let $\alpha'$ be a geodesic from the endpoint of $\alpha$ to $q$. 
	As $W_P$ is convex in the Davis complex, $\alpha' \subset wW_P$.
	By \cite[Proposition~2.4.4]{Bjorner-Brenti} and \cite[Corollary~2.4.5]{Bjorner-Brenti}, $\alpha \alpha'$ is geodesic as $\overline{\alpha}$ is the unique minimal length coset representative of $wW_\Gamma$.
	Moreover, as $\alpha\alpha'$ and $\gamma$ share endpoints, the words $\overline{\alpha \alpha'}$ and $w$ are equal as elements of $W_\Gamma$. 
	Recall that $ws$ and $wt$ are both geodesic words by the construction of a fan.
	Consequently, no expression for $w$ can end in $s$ or $t$.
	On the other hand, if $\alpha'$ is not a length $0$ geodesic, then $\overline{\alpha \alpha'}$ must end with either an $s$ or a $t$ (as $\alpha'$ is contained in a coset of $W_P$).
	Thus, $\alpha'$ must indeed have length $0$, and we must have that $\alpha = \gamma$. So, $\gamma$ is a minimal length path from $b$ to $W_P$.
	Again by \cite[Proposition~2.4.4]{Bjorner-Brenti}, we can conclude that $\gamma \gamma'$ is geodesic.
\end{proof}

\subsection{Filter construction}

For this subsection,  fix $W_\Gamma$ a wide-spherical-avoidant Coxeter group and $\Sigma$ its Davis complex. 
Let $\alpha = \alpha_0 \alpha_1 \dots$ and $\beta = \beta_0 \beta_1 \dots$ be geodesic rays in $\Sigma$ with a common basepoint $b$.
Here, $\alpha_i$ and $\beta_i$ are edges.
We will construct a \emph{filter $\mathcal{F}$ spanning $\alpha$ and $\beta$}: where $\mathcal F$ is a certain unbounded van-Kampen diagram over $\Sigma$ with oriented edges and whose boundary has image $\alpha \cup  \beta$.
We define $\mathcal F$ by giving a filtration $\mathcal{F}_0 \subset \mathcal{F}_1 \subset \dots$. 
For each $i$, we also define a spanning tree $T_i \subset \mathcal F_i$.

We begin with a van-Kampen diagram $\mathcal{F}_0$ consisting of a single vertex $b$, and two rays based at $b$ whose images in $\Sigma$ are $\alpha$ and $\beta$ respectively. By abuse of notation, we denote the edge in $\mathcal{F}_0$ corresponding $\alpha_i$ (resp. $\beta_i$) also by $\alpha_i$ (resp. $\beta_i$). 
Note that $\mathcal{F}_0$ does not contain any $2$-cells.

Next, we define $\mathcal F_1$. By Lemma~\ref{lem:cox_fan_existence}, there is a fan with length $0$ base path $b$ and with $\alpha_0$ and $\beta_0$ as its left and right fan edges respectively. Fix such a fan $F$ and call it the \emph{level $1$ fan}.
Set 
\[\mathcal F_1 = F \bigsqcup_{\alpha_0 \cup \beta_0} \mathcal F_0\]
We orient the edges of $\alpha$ and $\beta$ by the orientation of these paths, and we orient the edges of $F$ by the orientation on $F$.
It readily follows that these choices give a consistent orientation on the edges of $\mathcal F_1$.
Note also that the endpoint of each fan edge of $F$ is adjacent to exactly two outgoing edges in $\mathcal F_1$.
We define the tree $T_1$ to be $\mathcal F_1 \setminus E$, where $E$ is the set of top-left edges of $F$. 
It readily follows that $T_1$ is a spanning tree in $\mathcal F_1$.

Suppose that we have defined the van-Kampen diagram $\mathcal F_{n-1}$, corresponding level $n-1$ fans and a spanning tree $T_{n-1}$.
Let $F_1, \dots, F_r$ be all level $n-1$ fans and let $\gamma_1, \dots, \gamma_r$ be their corresponding base paths. 
For each $1 \le i \le r$, we further assume that $\gamma_i$ is contained in $T_{n-1}$ and that the endpoint of each fan edge of $F_i$ is adjacent to exactly two outgoing edges in $\mathcal F_{n-1}$.
We now define $\mathcal F_n$.

Fix $1 \le l \le r$, and let $f_0, \dots, f_k$ be the fan edges of $F_l$. 
Let $v_i$ be the endpoint of $f_i$, and let $x_i$ and $y_i$ be the two outgoing edges adjacent to $v_i$. 
We assume that $x_i$ is left of $y_i$ with respect to the orientation of the path $\gamma_l f_i$.
Note that $\gamma_lf_ix_i$ (resp. $\gamma_lf_iy_i$) is geodesic as either it lies on $\alpha$, $\beta$ or, otherwise,  Lemma~\ref{lem:cox_fan_geodesics} applies.
By Lemma~\ref{lem:cox_fan_existence} there exists a fan $F_{l,i}$ with left fan edge $x_i$, right fan edge $y_i$ and base path $\gamma_l f_i$. 
We call each $F_{l,i}$ (for values of $i$ and $l$ where defined) a \emph{level $n$ fan}.
We define $\mathcal F_n$ to be the van-Kampen diagram that is $\mathcal F_{n-1}$ together with the \emph{level $n$ fans}. 
The orientation on edges of $\mathcal F_n$ is given by the orientation of $\mathcal F_{n-1}$ together with the orientations on the level $n$ fans.
The spanning tree $T_n$ is obtained from the $1$-skeleton of $\mathcal{F}_n$ by removing the top left edges of all fans of all levels.

Let $f$ be a fan edge of $F_{l,i}$. We claim that the endpoint $u$ of $f$ has exactly two outgoing edges in $\mathcal F_n$. 
If $f$ is an interior fan edge, then this immediately follows by the construction of a fan.
Suppose then that $f$ is a left fan edge.
If $f$ is contained in $\alpha$ then $u$ has an outgoing edge that is on $\alpha$ and an outgoing edge on $F_{l,i}$, and we are done in this case. (See left magenta fan in Figure \ref{fig_filter-induction}.)
A similar argument follows if $f$ is contained on $\beta$. (See right magenta fan in Figure \ref{fig_filter-induction}.)
On the other hand, if $f$ is not contained on $\alpha \cup \beta$, then it is contained in a polygon $c$ that is part of a fan of level less than $n$. 
There are two possibilities. If $u$ is not the top vertex of $c$, then $f$ has an outgoing edge on $c$ and an outgoing edge on $F_{l,i}$, and the claim follows. (See blue fan in Figure \ref{fig_filter-induction}.)
On the other hand, suppose that $u$ is the top vertex of $c$.
Let $f' \neq f$ be the other edge of $c$ that is adjacent to $u$. 
It follows that $f'$ is either the left or right fan edge of a level $n$ fan $F'$.
Thus, $u$ has an outgoing edge in $F$ and another in $F'$. (See left side of red fan in Figure \ref{fig_filter-induction}.)
\begin{figure} \label{fig_filter-induction}
    \begin{overpic}[percent]{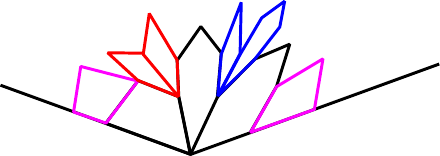}
     \put(0,10){$\alpha$}
     \put(100,15){$\beta$}
  \end{overpic}
  \caption{Filter induction}    
  \end{figure}

The filter $\mathcal F$ is defined to be the direct limit of $\mathcal F_0 \subset \mathcal F_1 \subset \dots$. The \emph{spanning tree of $\mathcal F$} is defined to be the direct limit of the trees $T_0 \subset T_1 \subset \dots$. A \emph{fan of $\mathcal F$} is one of the level $i$ fans for some $i$.

For the remainder of this section, we fix $\alpha = \alpha_0 \alpha_1 \dots$ and $\beta = \beta_0 \beta_1 \dots$ be geodesic rays in $\Sigma$ with a common basepoint $b$, and we fix a filter $\mathcal F$ spanning $\alpha$ and $\beta$. We also fix the spanning tree $T$ of $\mathcal F$.

\begin{lemma} \label{lem:cox_filter_geodesics}
	The image in $\Sigma$ of a directed path in $\mathcal F$ is geodesic. 
\end{lemma}
\begin{proof}
	The claim follows from induction, the definition of a fan, and Lemma \ref{lem:cox_fan_geodesics}.
\end{proof}

The next lemma easily follows from the construction of a filter. 
\begin{lemma} \label{lem:incoming_and_outgoing_edges}
	Let $v$ be a vertex of $\mathcal F$. Then 
	\begin{itemize}
	    \item The vertex $v$ has exactly two incoming edges if it is the top vertex of a polygon, and it has exactly one incoming edge otherwise.
	    \item There is a unique fan of $\mathcal F$ that contains all outgoing edges at $v$.
	\end{itemize}
\end{lemma}

\begin{definition}
    Given a vertex $v \in \mathcal F$, the \emph{fan at $v$} is the unique fan of $\mathcal F$ that contains all outgoing edges at $v$. This fan exists and is well-defined by Lemma~\ref{lem:incoming_and_outgoing_edges}.
\end{definition}

\begin{definition}[Itineraries]
Let $e$ be a directed edge of $\mathcal F$, and let $v$ be its startpoint. Let $F$ be the fan at $v$. 
We say that $e$ is an $L$-edge (resp $R$-edge, $I$-edge) of $\mathcal F$ if $e$ is a left (resp. right, internal) fan edge of $F$.

Let $\gamma$ be a finite or infinite directed path in the filter $\mathcal F$ with edges 
 $e_0, e_1, \dots $.
 The \emph{itinerary} of $\gamma$ is the infinite sequence $(x_0, x_1, \dots)$, with each $x_j \in \{L, R, I\}$, such that $x_{i}$ is equal to $L$ (resp. $R$, $I$) if $e_{i}$ is an $L$-edge (resp. $R$-edge, $I$-edge). 
\end{definition}

For the next lemma, recall from Definition~\ref{def:constants} that the constant $V_\Gamma$ is the number of vertices of $\Gamma$ and that $M_\Gamma$ is the smallest constant such that $W_{\text{Label(w)}}$ is infinite for any geodesic word $w$ with $|w| > M_\Gamma$.

\begin{lemma} \label{lem:cox_internal_edge_bound}
	Let $\gamma$ be a directed path in $T$ whose itinerary contains at least $M_{\Gamma} + V_\Gamma+1$ $I$-edges. 
	Then $\text{Label}(\gamma)$
	is not contained in a wide subgraph of $\Gamma$.
\end{lemma}
\begin{proof}
	Suppose for a contradiction that $\text{Label}(\gamma) \subset \Delta \subset \Gamma$ for some wide subgraph $\Delta \subset \Gamma$.
	Let $e_1, \dots, e_{V+1}$ be the last $V+1$ $I$-edges of $\gamma$, indexed with respect to the orientation of $\gamma$. 
	In other words, if $e$ is an $I$-edge of $\gamma$ then either $e = e_i$ for some $i$, or $e$ lies before $e_1$ along the orientation of $\gamma$. 
	Let $v_i$ be the startpoint of $e_i$, and let $F_i$ be the fan at $v_i$.
	Let $\gamma_i$ be the initial subpath of $\gamma$ up to $v_i$.
	
	Each path $\gamma_i$ has length larger than $M_\Gamma$ and is contained in some wide subgraph.
    Thus, by the definition of a fan and as $e_i$ is an $I$-edge, the label of $e_i$ is not in $\text{Label}(\gamma_i)$.
	However, this implies $\text{Label}(\gamma)$ contains at least $V_\Gamma+1$ distinct vertices of $\Gamma$. This is a contradiction as $|V(\Gamma)| = V_\Gamma$.
\end{proof}

\begin{lemma} \label{lem:cox_poly_base}
	Let $c$ be a polygon in $\mathcal F$. Let $\lambda$ and $\rho$ be respectively the left and right paths of $c$. 
	Let $l$ be the first edge of $\lambda$ and $r$ be the first edge of $\rho$.
	If $r$ is an $R$-edge, then $l$ is an $I$-edge. Similarly, if $l$ is an $L$-edge, then $r$ is an $I$-edge.
\end{lemma}
\begin{proof}
Let $v$ be the unique vertex of $c$ that contains both $l$ and $r$.
Let $F$ be the fan at $v$, and note that $l, r \subset F$.
As $\mathcal F$ is planar, $c \subset F$.
The claims in the lemma now follow by recalling that, by the definition of a fan, $F$ contains at least $3$ fan edges.
\end{proof}

\begin{definition}
    An \emph{$LR$-path} is a directed length $2$ path in $\mathcal F$ whose first edge is an $L$-edge and whose second edge is an $R$-edge. An \emph{$LR$-subpath} of a path, is a subpath that is an $LR$-path.
\end{definition}

\begin{lemma} \label{lem:LR}
    Let $\gamma$ be an $LR$-path in $T$. 
    Let $F$ be the the fan at the startpoint of $\gamma$. Then, $\gamma$ is contained in a polygon in $F$.
\end{lemma}
\begin{proof}
Suppose that $F$ is a fan of level $n$.
Let $e$ and $f$ be respectively the first and second edges of $\gamma$.
As $e$ is a left fan edge of $F$, there is a unique polygon $c$ of $F$ that contains $e$. 
As $f$ is an $R$-edge, the fan of level $n+1$ at the endpoint of $e$ has $f$ as its right fan edge. 
Thus, $f$ must be contained in $c$.
\end{proof}

\begin{lemma} \label{lem:cox_RL_edge_bound}
	Let $\gamma$ be a directed path in $T$ containing at least $M_\Gamma + V_\Gamma + 1$ $LR$-subpaths.
	Then, $\text{Label}(\gamma)$ is not contained in a wide subgraph of~$\Gamma$.
\end{lemma}
\begin{proof}
	Suppose for a contradiction that $\text{Label}(\gamma) \subset \Delta \subset \Gamma$ for some induced wide subgraph $\Delta \subset \Gamma$.
	Let $e_1f_1, \dots, e_{l}f_{l}$ be the last $l := M_\Gamma + V_\Gamma +1$ $LR$-subpaths of $\gamma$ with respect to $\gamma$'s orientation and ordered by $\gamma$'s orientation.
	Clearly these subpaths are disjoint by definition of $LR$-subpaths.
	Let $v_i$ be the startpoint of $e_i$, and let $F_i$ be the fan at $v_i$. 
	Let $\gamma_i$ be the initial subpath of $\gamma$ up to $v_i$.
	
	Fix $i > M_\Gamma$, and note that $|\gamma_i| > M_\Gamma$.
	Let $c$ be the polygon in $F_i$ containing both $e_i$ and $f_i$ which exists by Lemma~\ref{lem:LR}.
	Let $e$ be the edge of $c$ that is distinct from and incident to $e_i$. 
	By Lemma~\ref{lem:cox_poly_base}, $e$ is an $I$-edge.
	As $|\gamma_i| > M_\Gamma$ and as $\gamma_i \subset \Delta$, it follows from the definition of a fan that the label of $e$ is not contained in $\text{Label}(\gamma_i)$.
	Moreover, $f_i$ and $e$ have the same label, as they are contained in a common polygon and are separated by an edge in this polygon.
	From this we deduce that the labels of the edges $e_{M_\Gamma}, \dots, e_l$ are all distinct. However, this is a contradiction as there are at least $V_\Gamma+1$ such edges while there are at most $V_\Gamma$ possible labels.
\end{proof}

The next lemma easily follows from the construction of a filter.
\begin{lemma} \label{lem:polygon_edges}
Let $c$ be a polygon in $\mathcal F$, and let $\lambda$ and $\rho$ be its left and right paths respectively. 
Then every edge of $\lambda$ (resp. $\rho$), except its first edge, is an $R$-edge (resp. $L$-edge).
\end{lemma}

\begin{lemma} \label{lem:top_fan}
Let $c$ be a polygon in $\mathcal F$.
Then there is a polygon $d$ such that
the first edge of $d$'s right path is equal to the last edge of $c$'s left path.
\end{lemma}
\begin{proof}
Suppose that $c$ is a $2m$-gon.
Let $\lambda_1, \dots, \lambda_m$ be the edges of the left path of $c$, ordered with respect to the orientation of these paths.
Suppose that the fan containing $c$ is of level $n$.
As $\lambda_1$ is a fan edge of a level $n$ fan, it follow from the definition of a filter and by Lemma~\ref{lem:polygon_edges} that for each $2 \le i \le m$ there is a fan $F_{i}$ of level $n + i -1$ containing $\lambda_i$ as its right fan edge.   
Let $d$ be the polygon in $F_m$ containing $\lambda_m$, and let $\gamma$ be its right path.
The first edge of $\gamma$ is an $R$-edge equal to $\lambda_m$.
\end{proof}

For the next lemma, recall from Definition~\ref{def:constants} that $R_\Gamma = \max\{m_{st} ~ | ~ (s, t) \in E(\Gamma)\}$.
\begin{lemma} \label{lem:cox_L_bound}
	Let $\gamma$ be a directed path in $T \setminus (\alpha \cup \beta)$ of length at least $R_\Gamma(M_\Gamma + V_\Gamma +2)$ that contains only $L$-edges.
	Then, $\text{Label}(\gamma)$ is not contained in a wide subgraph of~$\Gamma$.
\end{lemma}
\begin{proof}
    Let $c_1$ be the polygon whose right path contains the first edge of $\gamma$. Such a polygon exists as $\gamma \cap \alpha = \emptyset$. 
    Let $\rho_1$ be the right path of $c_1$.
    By Lemma~\ref{lem:polygon_edges}, an initial subpath $\rho_1'$ of $\gamma$ is a terminal subpath of $c_1$'s right path.
    Let $v$ be the endpoint of $\rho_1'$.
    By Lemma~\ref{lem:top_fan} and as $|\gamma| > 2R_\Gamma$, there is a polygon $c_2$ with right path $\rho_2$ such that the second vertex of $\rho_2$ is $v$ and all edges of $\rho_2$, except its first edge, are contained in $\gamma$.
    Proceeding in this way, we see that there are polygons $c_1, \dots, c_l$ with corresponding right paths $\rho_1, \dots, \rho_l$ such that $\gamma = \rho_1'  \dots \rho_{l}'$ and the following hold:
	\begin{itemize}
	    \item For each $1 < i < l$, $\rho_i'$ is equal to $\rho_i$ minus its first edge. 
	    \item $\rho_1'$ is a terminal subpath of $\rho_1$ and $\rho_l'$ is an initial subpath of $\rho_l$.
	    \item There is a dual curve $H$ in $\mathcal F$ that is dual to the first edge of $\rho_i$ for each $i \ge 1$ and whose carrier contains $\gamma$.
	\end{itemize}
	
	Let $\gamma'$ be the path opposite to $\gamma$ with respect $Q$.
	By Lemmas~\ref{lem:cox_poly_base},~\ref{lem:top_fan},~\ref{lem:polygon_edges} and by the first two items above, $\gamma'$ contains $l-1$ $I$-edges. 
	As $|\gamma| \ge R_\Gamma(M_\Gamma + V_\Gamma +2)$,  $l \ge M_\Gamma + V_\Gamma +2$.  
    Thus, by Lemma~\ref{lem:cox_internal_edge_bound}, $\text{Label}(\gamma')$ is not contained in an induced wide subgraph of $\Gamma$.
    As $\gamma'$ is opposite to $\gamma$ with respect to $Q$, they have the same label, and so $\text{Label}(\gamma)$ is also not contained in an induced wide subgraph of $\Gamma$.
\end{proof}

\begin{lemma} \label{lem:cox_R_bound}
	Let $\gamma$ be a directed path in $T \setminus (\alpha \cup \beta)$ containing only $R$-edges. Then $\gamma$ has length at most $R_\Gamma$.
\end{lemma}
\begin{proof}
    Let $e$ be the first edge of $\gamma$, and let $F$ be the fan at the startpoint of $e$. 
    Let $c$ be the polygon that contains $e$ and is not contained in $F$.
    Note that $c$ exists as $\gamma \cap \beta = \emptyset$.
    It follows that $e$ is contained on the left path of $c$. By Lemma~\ref{lem:polygon_edges} and as $T$ does not contain the top left edges of polygons, $\gamma$ is a subpath of the left path of $c$. As the left path of $c$ has length at most $R_\Gamma$, the lemma follows. 
\end{proof}

The next proposition, using the results from this section, allow us to conclude that large enough directed paths in $T \setminus (\alpha \cup \beta)$ avoid wide subgraphs. 
\begin{proposition} \label{prop:cox_not_in_product_regions}
	There is a constant $N$, depending only on $M_\Gamma$, $R_\Gamma$ and $V_\Gamma$, such that any directed path $\gamma$ in $T \setminus (\alpha \cup \beta)$ 
	of length larger than $N$ does not have $\text{Label}(\gamma)$ contained in a wide subgraph.
\end{proposition}
\begin{proof}
    Let $\gamma$ be a path in $T \setminus (\alpha \cup \beta)$ such that $\text{Label}(\gamma)$ is contained in a wide subgraph of $\Gamma$.
	By Lemmas \ref{lem:cox_internal_edge_bound}, \ref{lem:cox_RL_edge_bound}, \ref{lem:cox_L_bound} and \ref{lem:cox_R_bound}, we have that $\gamma$:
	\begin{enumerate}
		\item Contains at most $M_\Gamma + V_\Gamma +1$ $I$-edges
		\item Contains at most $M_\Gamma + V_\Gamma +1$ $LR$-subpaths.
		\item Contains no subpath with $R_\Gamma(M_\Gamma + V_\Gamma +2)$ consecutive $L$-edges.
		\item Contains no subpath with $R_\Gamma$ consecutive $R$-edges.
	\end{enumerate}
	Let $L$ be the length of $\gamma$.
	By (1) and (2), there is a subpath $\gamma'$ of $\gamma$ of length at least $\frac{L - 3(M_\Gamma + V_\Gamma +1)}{2(M_\Gamma + V_\Gamma +1)}$ that contains no $I$-edges and no $LR$-subpaths.
	By (4), all edges of $\gamma'$ must be $L$-edges, except maybe its first $R_\Gamma$ edges.
	Thus, $\gamma'$ contains a subpath $\gamma''$ with at least $\frac{L - 3(M_\Gamma + V_\Gamma +1)}{2(M_\Gamma + V_\Gamma +1)} - R_\Gamma$ consecutive $L$-edges.
	By (3) we now have that
	\[\frac{L - 3(M_\Gamma + V_\Gamma +1)}{2(M_\Gamma + V_\Gamma +1)} - R_\Gamma \le |\gamma''| \le  R_\Gamma(M_\Gamma + V_\Gamma +2)\]
    Solving for $L$, we get an upper bound on the length of $\gamma$ that depends only on $M_\Gamma$, $R_\Gamma$ and $V_\Gamma$. The proposition then follows.
\end{proof}

\section{Connectivity and local connectivity of Morse boundary} \label{sec:connected}

In this section, we show both connectivity and local connectivity of the Morse boundary of an affine-free, wide-spherical-avoidant 
Coxeter group. 

\subsection{Connectivity of the Morse boundary}
In this subsection, we prove connectivity of Morse boundaries.

By Lemma~\ref{lem:cox_filter_geodesics}, geodesics from the base point $o$ of a filter $\mathcal{F}$ determine geodesics in $\Sigma$ with the same edge labels. 
Consequently, there is a natural proper map $f \colon \mathcal{F} \to \Sigma^1$ where $\Sigma^1$ is the $1$-skeleton of the Davis complex. 
Recall that $\Sigma^1$ is the Cayley graph of $W_\Gamma$ with respect to the generating set $V(\Gamma)$. We show that geodesics in the spanning tree of a filter  
are uniformly Morse, given that the filter's boundary geodesics are Morse:

\begin{lemma} \label{lem:filter geodesics Morse}
Let $W_\Gamma$ be an affine-free, wide-spherical-avoidant Coxeter group. Let $\alpha$ and $\beta$ be $N$-Morse geodesic rays in $\Sigma$ with common basepoint $o$. 
Let $\mathcal{F}$ be a filter associated to $\alpha$ and $\beta$, and let $T \subset \mathcal{F}$ its spanning tree. 
Then the image in $\Sigma^1$ of every geodesic in $T \cup \alpha \cup \beta$ with basepoint $o$ is $M$-Morse for some $M$ 
depending on $N$ and $\Gamma$.
\end{lemma}
\begin{proof}
By Proposition \ref{prop:cox_not_in_product_regions} and Theorem \ref{thm:morse_char}, every geodesic in $T\setminus \{\alpha, \beta\}$ is \
$M'$-Morse where $M'$ depends only on $\Gamma$. By setting $M'' = \max\{N, M'\}$, we have that every geodesic in $T$ is $M''$-morse, except possibly those that 
are a concatenation of a subset of $\alpha$ (resp. $\beta$) and a geodesic in $\mathcal F \setminus (\alpha \cup \beta)$. However, such geodesics are $M$-Morse for some
 $M$ depending only on $M''$ and $N$. 
\end{proof}

Recall that a subset of a metric space is quasi-convex if there exists a constant $K$ so that any geodesic between two points of the subset 
is in the $K$-neighborhood of of the subset. 
Moreover, the subset is \emph{stable} if it is quasi-convex and any two points of the subset can be connected by an $M$-Morse quasi-geodesic. 
We show that filters whose boundary geodesics are Morse are stable.

\begin{lemma} \label{lem:filter stable}
Let $\mathcal F$ be a filter in the Coxeter group $W_\Gamma$ whose boundary geodesics are $N$-Morse.
The image $f(\mathcal F)$ of $\mathcal{F}$ in $\Sigma^1$ is $M$-stable, where $M$ depends only on $N$ and $\Gamma$.
\end{lemma}

\begin{proof}
Let $x, y \in f(\mathcal F)$, and let $[x,y]$ be a geodesic in $\Sigma^1$ joining $x$ and $y$. Let $b$ be the image of the basepoint of $\mathcal F$, and let
 $[b, x]$ (resp. $[b, y]$) be geodesics from $b$ to $x$ (resp. $y$). 
By Lemma~\ref{lem:filter geodesics Morse}, $[b,x]$ and $[b,y]$ are $M'$-Morse for some $M'$ depending only on $N$ and $\Gamma$. 
By Lemma 2.3 of \cite{Cordes2017}, the geodesic $[x,y]$ is $M$-Morse where $M$ only depends on $M'$.

We are left to show quasi-convexity by showing that $[x,y]$ remains in a bounded neighborhood of $f(\mathcal F)$. 
By Lemma 3.5 of \cite{CordesHume}, the geodesic triangle $[b,x] \cup [x,y] \cup [b, y]$ is $K$-slim, where $K$ depends only on the Morse constants 
of these geodesic segments. 
Thus, $f(\mathcal{F})$ is $K$-quasi-convex as $[b,x] \cup [b, y] \subset f(\mathcal F)$.
\end{proof}

We are now ready to give necessary conditions for the connectivity of the Morse boundary:

\begin{theorem} \label{thm:MB is connected}
    The Morse boundary of an affine-free, wide-spherical-avoidant Coxeter group is connected.
\end{theorem}

\begin{proof}
Let $\Sigma^1$ be the $1$--skeleton of the Davis complex of an affine-free, wide-spherical-avoidant Coxeter group, and 
let $\alpha$ and $\beta$ be Morse geodesic rays in $\Sigma^1$ with common basepoint $o$. 
As the Coxeter group is wide-spherical-avoidant, there is a filter $\mathcal{F}$ with boundary geodesics $\alpha$ and $\beta$. 
Let $f: \mathcal{F} \to \Sigma^1$ be the natural map. 
Since $f$ is proper, by Proposition I.8.29 of \cite{BH}, $f$ induces a continuous map 
\[\mathrm{Ends}(\mathcal{F}) \to \mathrm{Ends}(f(\mathcal{F})) \subset \mathrm{Ends}(\Sigma^1)\]
Since $\mathcal{F}$ is one ended, we conclude that $f(\mathcal{F})$ is $1$-ended. 

By Lemma \ref{lem:filter stable}, $f(\mathcal{F})$ is stable and thus $f(\mathcal{F})$ is hyperbolic as a subset of $\Sigma^1$. 
Thus by Exercise III.H.3.8 of  \cite{BH} there exists as a continuous map $\partial(f(\mathcal{F})) \to \mathrm{Ends}(f(\mathcal{F}))$ 
where the fibers of this map are the connected components of the boundary. Thus  $\partial(f(\mathcal{F}))$ is a single connected component. 
Using again that $f(\mathcal{F})$ is stable, $\partial_* f(\mathcal{F})$ topologically embeds in $\partial_* W_\Gamma$, thus $\alpha$ and $\beta$ 
lie in the same connected component of the Morse boundary. As $\alpha$ and $\beta$ were arbitrary Morse geodesics, we conclude that 
the Morse boundary of $\Sigma^1$ is connected.
\end{proof}

\subsection{Local connectivity of Morse boundary}
We now focus on the local connectivity of Morse boundaries. We start with the following lemma that
 shows how to extend a geodesic segment to a Morse geodesic.
\begin{lemma} \label{lem:geo_extension}
Let $W_\Gamma$ be an affine-free, wide-spherical-avoidant Coxeter group, and let $\Sigma$ be its Davis complex.
Let $\gamma$ be a geodesic segment in $\Sigma^1$. Then, there exists a $M$--Morse geodesic $\alpha$ such that $\gamma \alpha$ is geodesic, 
where $M$ only depends on $\Gamma$.
\end{lemma}
\begin{proof}
    Let $w$ be the label of $\gamma$, and let $K \subset V(\Gamma)$ be the set of letters that an expression for $w$ can end with. Recall that $W_K$ is finite and that $ws$ is geodesic for any $s \in \Gamma \setminus K$.
    Let $\gamma'$ be the wide tail of $\gamma$, and let $\Delta$ be a wide subgraph such that $\text{Label}(\gamma') \subset \Delta$. 
    
    As $\Gamma$ is wide-spherical-avoidant, there is a vertex $s \in \Gamma \setminus (\Delta \cup K)$.
    Let $e$ be the edge adjacent to $\gamma$ and labeled by $s$. 
    It follows that $\gamma \cup e$ is geodesic.
    By iteratively applying this argument, we can extend $\gamma$ to an infinite geodesic ray $\gamma \alpha$.
    By the same argument as that of Lemma~\ref{lem:cox_internal_edge_bound}, there is a constant $C$ depending only on $\Gamma$
     such that any subpath $\alpha'$ of $\alpha$ of length larger than $C$ does not have $\text{Label}(\alpha')$ contained in a wide subgraph. 
    By Theorem \ref{thm:morse_char}, $\alpha$ is Morse, where the constant only depends only $C$ and consequently only on $\Gamma$.
\end{proof}
In order to show local connectivity, we define the concept of a multi-tail filter, constructed by gluing 
together a finite number of filters.

\begin{definition}[Multi-tail filter] \label{def:MTF} 
Let $\alpha$ and $\beta$ be $M$-Morse geodesics in the Davis complex of a Coxeter group based at some vertex $b$. Let $n \in \mathbb{N}$, and let 
$\sigma$ be a geodesic from $\alpha(n)$ to $\beta(n)$. 
Let $\gamma_0, \gamma_1, \dots, \gamma_k$ be geodesic segments from $b$ to $\sigma$ which are tails to filters 
$\mathcal F_0, \dots, \mathcal F_k$. Moreover, suppose that $\mathcal F_i$ is a filter between $M$-Morse geodesics $\alpha_i$ and $\beta_i$. 
Additionally, 
suppose that $\alpha_0 = \alpha$, $\beta_k = \beta$ and, for each $i < k$, $\beta_i$ is equal to $\alpha_{i+1}$ except 
possibly outside a finite set. 
We call $\mathfrak F$ a \emph{multi-tailed filter of level $n$} between $\alpha$ and $\beta$. 
Moreover, we say that each $\gamma_i$ is a \emph{tail} of $\mathfrak F$.  
\end{definition}

\begin{figure} \label{mulitail}
    \begin{overpic}[scale=.6]{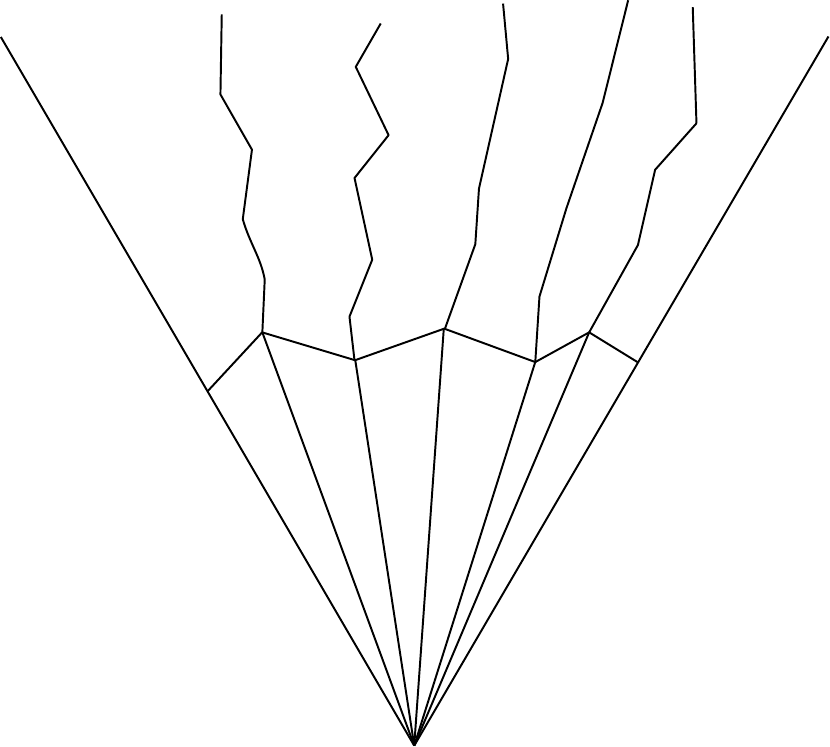}
            \put(47, 0){$b$}
            \put(19, 41){$\alpha(n)$}
            \put(78, 45){$\beta(n)$}
            \put(21, 55){\LARGE $\mathcal{F}_0$}
            \put(35, 60){\LARGE $\mathcal{F}_1$}
            \put(80, 62){\LARGE $\mathcal{F}_k$}
            \put(0,86){$\alpha_0=\alpha$}
            \put(99,86){$\beta_k=\beta$}
            \put(25,89){$\beta_0=\alpha_1$}      
  \end{overpic}
  \caption{Multi-tailed filter of level $n$}    
  \end{figure}
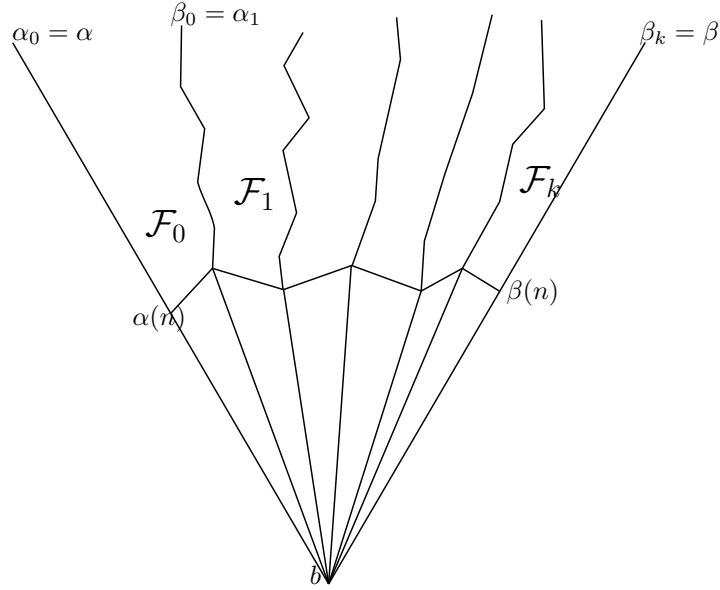

We show that multi-tail filters of level $n$ always exist between Morse geodesics in an affine-free, wide-spherical-avoidant Coxeter group.
To do so, we first prove the following technical lemma.

\begin{lemma} \label{lem:multi-filter}
    Suppose that $W_\Gamma$ is an affine-free, wide-spherical-avoidant Coxeter group.
    Let $\alpha$ and $\beta$ be Morse geodesic rays in the Davis complex of $W_\Gamma$ based at the same point $b$. 
    Fix $n \in \mathbb N$, and let $\sigma$ be a geodesic from $\alpha(n)$ to $\beta(n)$. 
    Also, let $d = d(\alpha(n), \beta(n))$, $\gamma_0 = \alpha([0,n])$, $\gamma_d = \beta([0, n])$ 
    and $\gamma_i$ be a geodesic from $b$ to $\sigma(i)$ for $0 < i < d$.
    Let $e_i$ be the edge of $\sigma$ between $\sigma(i-1)$ and $\sigma(i)$.
    Then, there exist $M$-Morse geodesics $\alpha_0, \dots, \alpha_d$ such that
    \begin{enumerate}
        \item $\gamma_i\alpha_i$ is geodesic for each $0 \le i \le d$
        \item For each $1 \le i \le d$, either 
            \begin{enumerate}
                \item $\alpha_i = e_i\alpha_{i-1}$ or
                \item $\gamma_{i-1}e_i\alpha_i$ is geodesic
            \end{enumerate}
    \end{enumerate}
    Moreover, $M$ only depends on $\Gamma$, $d(\alpha(n), \alpha(n))$ and the Morse constants of $\alpha$ and~$\beta$.
\end{lemma}
\begin{proof}
    First set $\alpha_0 = \alpha([s, \infty))$, and note that the claim follows for $i=0$.
    We now prove the claim for $i = k$ assuming that it is true for $i = k-1$. 
    
    Let $H$ be the wall dual to $e_k$. 
    We claim that either $H$ intersects $\gamma_{k-1}$ and does not intersect $\gamma_k$, or 
    $H$ intersects $\gamma_k$ and does not intersect $\gamma_{k-1}$.
    To see this, consider a disk diagram $D$ with boundary path $\gamma_{k-1}e_k\gamma_k^{-1}$. 
    As with any disk diagram over the Davis complex, $H$ must be dual to an even number of edges of the boundary $D$. 
    Note that $H$ cannot be dual to $4$ or more edges, as then $H$ would either be dual to two edges the geodesic 
    $\gamma_{k-1}$ or to two edges of the geodesic $\gamma_k$.
    As $H$ is dual to $e_k$, it must then be dual to exactly two edges of the boundary of $D$.
    The claim now follows. 
    
    The proof is split into two cases, depending whether $H$ intersects $\gamma_{k-1}$ or $\gamma_{k}$.
    Suppose first that $H$ intersects $\gamma_{k-1}$.
    In this case, we claim that $\gamma_k e_k \alpha_{k-1}$ is geodesic. 
    Suppose for a contradiction that it is not a geodesic. 
    Then some wall $K$ is dual to two edges of $\gamma_k e_k \alpha_{k-1}$. 
    Note that $K \neq H$, for otherwise we would either have that $H$ intersects $\alpha_{k-1}$ 
    and so intersects the geodesic $\gamma_{k-1}\alpha_{k-1}$ twice, or we would have that $H$ intersects $\gamma_k$ 
    which cannot happen by the previous paragraph.
    Thus, $K$ must intersect both $\alpha_{k-1}$ and $\gamma_k$.
    we consider $K \cap D$ and see that $K$ must intersect either $\gamma_{k-1}$ or $e_k$. 
    However, the former is impossible as $\gamma_{k-1} \alpha_{k-1}$ is geodesic and the latter is impossible as $K \neq H$.
    Thus, $\gamma_k e_k \alpha_{k-1}$ is indeed geodesic. We can now set $\alpha_k := e_k \alpha_{k-1}$. 
    With this choice, clearly (1) and (2a) hold for $i = k$.
    Moreover, as $\alpha_{k-1}$ is $M'$-Morse for some $M'$, $\alpha_k$ is $M$-Morse with $M$ only depending on $M'$.
    
    On the other hand, suppose that $H$ intersects $\gamma_k$. 
    By Lemma~\ref{lem:geo_extension}, there is a Morse geodesic $\alpha_k$ so that $\gamma_k\alpha_k$ is geodesic. 
    Moreover, $M$ only depends on $\Gamma$.
    By an argument similar to the one in the previous case, $\gamma_{k-1} e_k \alpha_k$ is geodesic.  
    Consequently, (1) and (2b) hold for $i = k$ in this case. 
\end{proof}

We are now ready to show that multi-tail filters exist.

\begin{lemma} \label{lem:MF existance}
    Let $\alpha$ and $\beta$ be Morse geodesics based at a common point in the Davis complex of an affine-free, wide-spherical-avoidant Coxeter 
    group. Then for every $n \in \mathbb{N}$, there exists a multi-tail filter of level $n$ between $\alpha$ and $\beta$.
\end{lemma}
\begin{proof}
    Using Lemma~\ref{lem:multi-filter}, we will construct a multi-tailed filter. 
    We fix the notation from that lemma.
    We will choose a subsequence $i_0, \dots, i_k$ of the subsequence of $0, \dots, d$ 
    so that $\gamma_{i_0}, \dots \gamma_{i_k}$ are tails of filters $\mathcal F_0, \dots, \mathcal F_k$ which make up 
    a multi-tail filter of level $n$ between $\alpha$ and $\beta$.

    Let $j$ be the smallest integer (if it exists) such that (2b) from Lemma~\ref{lem:multi-filter} holds for $i = k$. 
    There are two cases depending whether $j$ exists or not. Suppose first that $j$ exists.
    We set $i_0 = j - 1$ and construct a filter between $\alpha_{i_0} = e_{j-1} e_{j-1} \dots e_1 \alpha_0$ and $\beta_{i_0} = e_j \alpha_j$.
    Note that by the same lemma, both $\gamma_{i_0}\alpha_{i_0}$ and $\gamma_{i_0} \beta_{i_0}$ are geodesic. 

    If no such integer exists, we set $i_0 = d$ and construct a filter between $\alpha_{i_0} = e_d e_{d-1} \dots e_1 \alpha_0$ 
    and $\beta_d$ with tail $\gamma_{i_0}$. In this case, $\gamma_{i_0}\alpha_{i_0}$ is geodesic from (2a) of the previous lemma
     and $\gamma_{i_0} \beta_{i_0}$ is geodesic as it is equal to $\beta$. 
   
    It is clear that we are done in the second case. Moreover, in the first case, we can repeat the process 
    by starting from the smallest integer larger than $j$ (if it exists) such that (2b) holds.
    By iterating in this manner, we construct the mult-tail filter.
\end{proof}

\begin{lemma} \label{lem:gamma morse}
    Let $\mathfrak{F}$ be an $n$-level multi-tailed filter between $N$-Morse quasi-geodesic rays $\alpha$ and $\beta$, and 
    let $f \colon \mathfrak{F} \to \Sigma$ the natural map.
    Then the tails of $\mathfrak F$ are $M$-Morse where $M$ depends only on $N$ and $d(\alpha(n), \beta(n))$.
\end{lemma}

\begin{proof}
Let $d=d(\alpha(n), \beta(n))$, and let $\sigma$ be the geodesic between $\alpha(n)$ and $\beta(n)$ associated with $\mathfrak F$. 
Let $\gamma_i$ denote the tail of $\mathfrak F$ from its basepoint to $\sigma(i)$.  
The concatenation $\gamma = \sigma([0,i]) \cup \gamma_i$ 
is a $(1,d)$-quasi-geodesic with endpoints on $\alpha$. 
Thus by Lemma 2.1 of \cite{Cordes2017} $\gamma'$ is bounded Hausdorff distance from $\alpha$, 
where the distance depends only on $N$ and $d$. 
By Lemma 2.5 of \cite{Charney-Sultan}, $\gamma$ is $M$-Morse where $M$ depends only on $N$ and $d$. 
\end{proof}

\begin{lemma} \label{lem:multi-tail filter stable}
Let $\mathfrak{F}$ be an $n$-level multi-tailed filter between $N$-Morse quasi-geodesic rays $\alpha$ and $\beta$, and 
let $f \colon \mathfrak{F} \to \Sigma$ the natural map.
Then $f(\mathfrak{F})$ is an $M$-stable subset where $M$ depends only on $N$ and $d(\alpha(n), \beta(n))$.
\end{lemma}

\begin{proof}
    By Lemma~\ref{lem:gamma morse} the image of the tails of $\mathfrak F$ are uniformly Morse, 
    with Morse constant only depending on $N$ and $d(\alpha(n), \beta(n))$.
    By Lemma ~\ref{lem:multi-filter} and Lemma ~\ref{lem:filter stable}, 
    the image of each of the filters making up $\mathcal{F}$ is stable, with 
    constant only depending on $N$ and $d(\alpha(n), \beta(n))$.
    Thus, by Lemma \ref{lem:concatenation of morse is morse}, all geodesics in the spanning tree of $\mathfrak{F}$ have $M$-Morse image 
    with $M$ only depending on $N$ and $d(\alpha(n), \beta(n))$.
    We can now follow the same proof as in Lemma \ref{lem:filter stable} and conclude that the image of $\mathfrak{F}$ is $M$-stable. 
\end{proof}

\begin{lemma} \label{lem:boundary of MTF is connected}
    Let $\alpha$ and $\beta$ be Morse geodesic rays, and let $\mathfrak{F}$ be a multi-tailed filter between $\alpha$ and $\beta$. 
    Then $\partial f(\mathfrak{F})$ is a connected subset of $\partial_* X$
    where $f \colon \mathfrak{F} \to \Sigma$ is the natural map. 
\end{lemma}

\begin{proof}
    From the proof of Theorem~\ref{thm:MB is connected}, 
    the boundary of the image of each filter making up $\mathfrak{F}$ is a connected subset of $\partial_* f(\mathfrak F)$. 
    As $f(\mathfrak F)$ is stable by the previous lemma, each quasi-geodesic in $f(\mathfrak F)$ must stay close to the image of one 
    of the filters in $\mathfrak F$. Thus, the Morse boundary of $\mathfrak F$ consists of the union of Morse boundaries of the filters 
    making up $\mathfrak F$. It follows that $\partial_* f(\mathfrak F)$ is connected.
\end{proof}

\begin{theorem} \label{thm:MB is locally connected}
The Morse boundary of an affine-free, one-ended, wide-spherical-avoidant Coxeter group is locally connected.
\end{theorem}

\begin{proof}
Let $\Sigma$ be the Davis complex of an affine-free, wide-spherical-avoidant Coxeter group.

To prove local connectivity of $\partial_* \Sigma$, we will show that $\partial_* \Sigma$ is \emph{connected im kleinen} for all $[\alpha] \in \partial_* \Sigma$, i.e., 
for every open neighborhood $V \subset \partial_* \Sigma$ of $[\alpha]$ there exists an open set $U \subset V$ containing $\alpha$ such 
that for any $\beta \in U$, there exists a connected set $K \subset V$ containing both $\alpha$ and $\beta$.
(For more information see \cite[Definition 27.14, Theorem 27.16]{Willard}.)

Let $[\alpha] \in \partial_* \Sigma$ be a point in the Morse boundary, and assume $\alpha$ an $M$--Morse geodesic ray representing $[\alpha]$ with basepoint $b$.
Let $V \subset \partial_* \Sigma$ be an open set containing $\alpha$. 
By definition of the topology on the Morse boundary, there exists an $n \in \mathbb{N}$ and $N \geq M$ so that the open set $V_n^N(\alpha)$ 
(as defined in Section~\ref{sec:background:morse boundary}) is contained in $\partial_*^N \Sigma \cap V$. 
Moreover, by Lemma~\ref{lem:multi-tail filter stable} and compactness, we can pick $N$ large enough so 
that any $3n$-level multi-tail filter between $\alpha$ and another $M$-Morse quasi-geodesic has 
$N$-stable image in $\Sigma$.
Define $U := V_{3n}^M(\alpha)$, and note that $U \subset V_n^N(\alpha) \subset V$.

Let $\beta$ be a point in $U$ and, by abuse of notation, also denote by $\beta$ an $M$-Morse geodesic ray with basepoint $b$ representing it. 
Let $\sigma$ be a geodesic joining $\alpha(3n)$ to $\beta(3n)$. 
Construct a multi-tailed filter $\mathfrak{F}$ corresponding to $\alpha$, $\beta$ and $\sigma$.
By our choice of $N$, $f(\mathfrak{F})$ is $N$-stable with $f: \mathfrak{F} \to \Sigma$ being the natural map.

We claim that
$\partial f(\mathfrak{F}) \subset V$. 
To show this, let $\tau$ be a quasi-geodesic ray in $f(\mathfrak{F})$ based at $b$ representing a point in $\partial f(\mathfrak{F})$.
Since $\alpha,~\beta \in U$ and since $N > M$, $d(\alpha(3n), \beta(3n)) < \delta_N$ with $\delta_N$ 
the constant associated to $N$ as in Section~\ref{sec:background:morse boundary}. 
By construction of a multi-tail filter, we also have that $d(\alpha(3n), \tau(3n)) < \delta_N$.
By Corollary 2.5 of \cite{Cordes2017}, $d(\alpha(t), \tau(t))< \delta_N$ for all $t \in [0, n]$. 
Thus, $\partial \mathfrak{F} \subset V_{n}^N(\alpha) \subset V$. 
The claim now follows as, by Lemma~\ref{lem:multi-tail filter stable}, 
$\partial f(\mathfrak{F})$ is a connected subset of $\partial_* \Sigma$. 
\end{proof}

\bibliography{refs}
\bibliographystyle{amsalpha}

\end{document}